\documentclass{amsart}

\usepackage{hyperref}   

\newtheorem{theorem}{Theorem}[section]

\newtheorem{proposition}[theorem]{Proposition}
\newtheorem{lemma}[theorem]{Lemma}
\newtheorem{corollary}[theorem]{Corollary}

\theoremstyle{definition}
\newtheorem{remark}[theorem]{Remark}
\newtheorem{dfn}[theorem]{Definition}

\newcommand{\Z}{\mathbb{Z}}   
\newcommand{\Q}{\mathbb{Q}}   

\newcommand{\gmu}{\boldsymbol{\mu}}

\newcommand{\CH}{\operatorname{CH}}
\newcommand{\Inv}{\operatorname{Inv}^3_{ind}}
\newcommand{\Sinv}{\operatorname{Inv}^3_{sd}}
\newcommand{\Br}{\operatorname{Br}}

\newcommand{\SB}{\operatorname{SB}}
\newcommand{\gPGL}{\operatorname{\mathbf{PGL}}}
\newcommand{\gSP}{\operatorname{\mathbf{SP}}}
\newcommand{\gSL}{\operatorname{\mathbf{SL}}}
\newcommand{\gSO}{\operatorname{\mathbf{SO}}}
\newcommand{\gPGSP}{\operatorname{\mathbf{PGSp}}}
\newcommand{\gSpin}{\operatorname{\mathbf{Spin}}}
\newcommand{\gHSpin}{\operatorname{\mathbf{HSpin}}}
\newcommand{\gPGO}{\operatorname{\mathbf{PGO}}}

\newcommand{\Sdec}{\operatorname{Sdec}}
\newcommand{\Dec}{\operatorname{Dec}}

\newcommand{\torsion}{\operatorname{tors}}

\newcommand{\ldeg}{\operatorname{ldeg}}
\newcommand{\hdeg}{\operatorname{hdeg}}
\newcommand{\wdeg}{\operatorname{wdeg}}
\newcommand{\lcm}{\operatorname{lcm}}

\title[The $K$-theory and cohomological invariants] 
{The $K$-theory of versal flags and cohomological invariants of degree 3}

\author
[S.~Baek] {Sanghoon Baek}

\address[Sanghoon Baek]{Department of Mathematical Sciences, 
KAIST,
291 Daehak-ro, Yuseong-gu,
Daejeon 305-701,
Republic of Korea}
\email{sanghoonbaek@kaist.ac.kr}
\urladdr{http://mathsci.kaist.ac.kr/~sbaek/}

\author
[R.~Devyatov]{Rostislav Devyatov}

\address[Rostislav Devyatov]{Department of Mathematics and Statistics, University of Ottawa, 585 King Edward Street, Ottawa, ON, K1N 6N5, Canada}
\email{deviatov@mccme.ru}
\urladdr{http://www.mccme.ru/~deviatov/}

\author
[K.~Zainoulline]{Kirill Zainoulline}

\address[Kirill Zainoulline]{Department of Mathematics and Statistics, University of Ottawa, 585 King Edward Street, Ottawa, ON, K1N 6N5, Canada}
\email{kirill@uottawa.ca}
\urladdr{http://mysite.science.uottawa.ca/kzaynull/}


\begin{document}

\begin{abstract}
Let $G$ be a split semisimple linear algebraic group over a field and let $X$ be a generic twisted flag variety of $G$. 
Extending the Hilbert basis techniques to Laurent polynomials over integers 
we give an explicit presentation of the Grothendieck ring $K_0(X)$ in terms of generators and relations in the case
$G=G^{sc}/\mu_2$ is of Dynkin type ${\rm A}$ or ${\rm C}$ (here $G^{sc}$ is the simply-connected cover of $G$);
we compute various groups of (indecomposable, semi-decomposable) cohomological invariants of degree 3, hence, generalizing and extending previous results in this direction.
\end{abstract}

\maketitle


\section{Introduction}

Let $G$ be a split semisimple linear algebraic group over a field $F$.
Let $U/G$ be a {\em classifying space} of $G$ in the sense of Totaro \cite[Rem.1.4]{To}, i.e. $U$ is an open
$G$-invariant subset in some representation of $G$ with $U(F)\neq
\emptyset$ and $U\to U/G$ is a $G$-torsor. Consider the generic fiber $U'$ of $U$ over $U/G$. It is a $G$-torsor over the quotient
field $F'$ of $U/G$ called the {\em versal} $G$-torsor \cite[Ch.I, \S
5]{GMS}. We denote by $X$ the respective flag variety $U'/B$ over
$F'$, where $B$ is a Borel subgroup of $G$, and call it the {\em
  versal} flag. The variety $X$ appears in many different contexts, e.g. related to cohomology of homogeneous $G$-varieties  (see \cite{GiZa} for an arbitrary oriented theory; Karpenko \cite{Kar}, \cite{Ka17}, \cite{Ka16} for Chow groups; Panin~\cite{Panin} for $K$-theory) and cohomological invariants of $G$ (see Merkurjev~\cite{Mer} and \cite{GaZa}, \cite{MNZ}). It can be viewed as a generic example of the so called {\em twisted flag variety}.

In the first part of the paper (Sections~2-4) we give an explicit presentation of the ring $K_0(X)$ in terms of generators modulo a {\em finite} number of relations 
in cases when $G=G^{sc}/\mu_2$, where $G^{sc}$ is the product of simply-connected simple groups of Dynkin types ${\rm A}$ or ${\rm C}$ and $\mu_2$ is a central subgroup of order 2.

Observe that for simply-connected $G$ the ring $K_0(X)$ can be identified with $K_0(G/B)$ (e.g., see Panin~\cite{Panin}), 
and by Chevalley theorem 
there is a surjective characteristic map $c\colon R(T_{sc}) \to K_0(G/B)$
from the representation ring of the split maximal torus $T_{sc}$ such that the kernel $\ker(c)=I_{sc}^W$ is generated
by augmented classes of fundamental representations. 
So, all relations in $K_0(X)$ correspond to $W$-orbits of fundamental weights.

If $G$ is not simply-connected (as in the $G^{sc}/\mu_2$-case), then the situation changes dramatically as by \cite[Ex.5.4]{GiZa} we have $K_0(X)\simeq R(T)/I_{sc}^W\cap R(T)$ and a finite set of generators of $I_{sc}^W\cap R(T)$ is not known in general. Note that by definition we have inclusions of abelian groups $I^W \subseteq I_{sc}^W\cap R(T) \subseteq  I_{sc}^W$ which all coincide if taken with $\mathbb{Q}$-coefficients. 
However, there are examples of semisimple groups 
(see \cite[Ex.3.1]{MNZ} and \cite{Bae}) where both quotients $I_{sc}^W\cap R(T)/I^W$ and 
$I_{sc}^W/I_{sc}^W\cap R(T)$ are non-trivial.

Our Theorem~\ref{thm:maingen} provides a complete list of generators (Definition~\ref{dfn:genlist}) of the ideal $I_{sc}^W\cap R(T)$. To prove this result we assume that the root system of $G^{sc}$ satisfies the generalized flatness condition (see Definition~\ref{dfn:genflatcond}).
In Section~4 we show that this condition holds for types ${\rm A}$ and ${\rm C}$.

In the second part of the paper  we study cohomological invariants of degree $3$ of $G$.
According to~Garibaldi-Merkurjev-Serre \cite[p.106]{GMS}, a degree~$d$ {\em cohomological invariant} is a natural transformation of functors $a\colon H^1(\,\cdot\,,G)\to H^d(\,\cdot\,,\Q/\Z(d-1))$
on the category of field extensions over $F$, 
where the functor $H^1(\,\cdot\,,G)$ classifies $G$-torsors, $H^d(\,\cdot\,,\Q/\Z(d-1))$ is the Galois cohomology. 
Following~Merkurjev~\cite{Mer}, an invariant is called {\em decomposable} if it is given by a cup-product of invariants of smaller degrees; 
the factor group of (normalized) invariants modulo decomposable is called the group of {\em indecomposable} invariants. For $d=3$ the latter
has been computed for all simple split groups  in \cite{Mer} and  \cite{BR}; 
for some semi-simple groups of type ${\rm A}$ in \cite{M16} and \cite{Bae};
for adjoint semisimple groups in \cite{Me16}.

Another key subgroup of {\em semi-decomposable} invariants introduced in \cite{MNZ} consists of invariants given by a cup-product of invariants up to some field extensions. For $d=3$ it coincides with the group of decomposable invariants for all simple groups \cite{MNZ}. It was also shown that these groups are different for $G=\gSO_4$ \cite[Ex.3.1]{MNZ} and for some semisimple groups of type ${\rm A}$ (see \cite{Bae}).

In Sections 6-11 we compute the groups of decomposable, indecomposable and semi-decomposable invariants of degree 3
for new examples of semisimple groups (e.g. $G^{sc}/\mu_2$, products of adjoint groups), hence, extending the results of \cite{Mer}, \cite{BR}, \cite{Bae},\ \cite{MNZ}, \cite{M16}. In particular, we essentially extend the examples \cite[Ex.3.1]{MNZ} and \cite{Bae}; we show that

$\bullet$ The factor group of semi-decomposable invariants of $G$ modulo decomposable is nontrivial if and only if $G$ is of classical type ${\rm A}$, ${\rm B}$, ${\rm C}$, ${\rm D}$. Moreover, we determine all the factor groups (and indecomposable groups) for an arbitrary product of simply-connected simple groups of the same Dynkin type modulo the central subgroups $\gmu_{2}$ (see Corollaries~\ref{cor:typeA}, \ref{cor:typeB}, \ref{cor:typec}, \ref{cor:typeD}, and Proposition \ref{prop:typeE}).

$\bullet$ If $G$ is of type ${\rm A}$, then the factor group of semi-decomposable invariants modulo decomposable (and the group of indecomposable invariants) can have an arbitrary order and contains any  homocyclic $p$-group (see Corollary \ref{cor:abelian}).

$\bullet$ If $G$ is of type ${\rm B}$ or ${\rm C}$, then it is always a product of cyclic groups of order 2 (see Corollaries~\ref{cor:typeB}, \ref{cor:typec}). 

\noindent
We also provide a direct proof of the fact that for the simple group $G=\gPGO_8$ any semi-decomposable invariant is decomposable (Corollary~\ref{pgo8}).

\paragraph{\bf Acknowledgements.} 
S.B. was partially supported by National Research Foundation of Korea (NRF) funded by the Ministry of Science, ICT and Future Planning (2016R1C1B2010037). The first author also would like to thank the university of Ottawa for the warm hospitality during his sabbatical in the fall of 2016. R.D. was partially supported by the Fields Institute for Research in Mathematical Sciences, Toronto, Canada. R.D. and K.Z. were partially supported by the NSERC Discovery grant RGPIN-2015-04469, Canada.

\section{Syzygies and divisibility for Laurent polynomials}

Let $\Lambda$ be a free abelian group of rank $n$ with a fixed basis $\{x_1,\ldots, x_n\}$.
Let $R$ be one of the rings $\Z$ or $\Z/m\Z$, $m\ge 2$.
Consider the group ring $R[\Lambda]$. It consists of finite linear combinations $\sum_j a_j e^{\lambda_j}$, $a_j\in R$, $\lambda_j\in \Lambda$.
We identify $R[\Lambda]$ with the Laurent polynomial ring $R[x_1^{\pm 1},\ldots,x_n^{\pm 1}]$ via 
$e^{x_i}\mapsto x_i$ and $e^{-x_i}\mapsto x_i^{-1}$. By a polynomial we mean always a Laurent polynomial, i.e., an element of $R[\Lambda]$.
We denote by $\Lambda_{i}$ a free subgroup with the basis $\{x_1,\ldots,x_i\}$, $1\le i < n$. Hence, $R[\Lambda_i]=R[x_1^{\pm 1},\ldots, x_i^{\pm 1}]$. 

\begin{dfn}\label{dfn:lpres}
Given $f\in R[\Lambda]$, we can express it uniquely as
\[
f=f_k x_n^k+f_{k-1} x_n^{k-1}+ \ldots + f_m x_n^m,\; \text{ where }f_i\in R[\Lambda_{n-1}],\; k,m\in \Z,\; k\ge m.
\]
The integer $k$ is called the \emph{highest degree of $f$ with respect to $x_n$} and denoted $\hdeg_{n}(f)$.
The integer $m$ is called the \emph{lowest degree of $f$ with respect to $x_n$} and denoted $\ldeg_{n}(f)$.
The difference $k-m$ is called the \emph{degree of $f$ with respect to $x_n$} and denoted $\wdeg_{n}(f)$.
\end{dfn}

By definition, if $\wdeg_n(f)=0$, then $f$ is a product of $x_n^k$ and a polynomial in $x_1, \ldots, x_{n-1}$.

\begin{dfn}
Let $f$, $p\in  R[\Lambda]$ and let $\ldeg_{n}(f)\ge d$ for some $d\in \Z$.
We say that it is possible to perform a {\em divison of $f$ by $p$ bounded by $d$} if there exist 
monomials $q$, $r\in R[\Lambda]$ such that
\begin{enumerate}
	\item $f=pq+r$.
	\item Either $r=0$ or ($\ldeg_n(r)\ge d$ and $\hdeg_n(r)<d+\wdeg_n(p)$).
\end{enumerate}
In this case $q$ is called the \emph{quotient}, and $r$ is called the \emph{remainder}.
\end{dfn}

\begin{dfn}
We call $p \in R[\Lambda]$ a {\em divisor} with respect to $x_n$ if it satisfies the following condition:

In the presentation of Definition~\ref{dfn:lpres}
\[
p=p_k x_n^k+
\ldots+
p_m x_n^m,\quad p_i\in R[\Lambda_{n-1}],\; k,m\in \Z,\; k\ge m,
\]
the leading coefficient $p_k$ is a {\em monic} monomial in $x_1,\ldots, x_{n-1}$.
\end{dfn}

\begin{lemma}\label{intdivision}
	Let $f$, $p\in  R[\Lambda]$ and let $\ldeg_{n}(f)\ge d$ for some $d\in \Z$.
	
	If $p$ is a divisor with respect to $x_n$, then it is possible to perform a division of $f$ by $p$ bounded by $d$.
\end{lemma}

\begin{proof}
	We proceed by induction on $\hdeg_n(f)$. If $\hdeg_n(f)<d+\wdeg_n(p)$, then we set $q=0$ and $r=f$.
	
	Suppose that $\hdeg_n(f)\ge d+\wdeg_n(p)$.
	Since $p$ is a divisor, we can write it as 
	\[
	p=Yx_n^{k}+p', \text{ where }Y\in R[\Lambda_{n-1}]\text{ is a monic monomial and}
	\]
	$p'$ is either $0$ or a polynomial with $\hdeg_n(p')<\hdeg_n(p)=k$ and $\ldeg_n(p')=\ldeg_n(p)$.
Observe that $Y$ is invertible in $R[\Lambda_{n-1}]$.

	We write $f$ as $f=gx_n^m+f'$, where $m=\hdeg_n(f)$, $g\in R[\Lambda_{n-1}]$, 
	and $f'$ is either $0$, or a polynomial with $\hdeg_n(f')<m$ and $\ldeg_n(f')=\ldeg_n(f)$.
	
	Set $q_0=gY^{-1}x_n^{m-k}$. Then $Yx_n^kq_0=gx_n^m$. If both $f'$ and $p'$ are $0$, then $Yx_n^k=p$ and $gx_n^m=f$, 
	so $pq_0=f$, and we are done.
	
	Consider the polynomial $f''=f'-q_0p'$. We have $\hdeg_n(q_0)=\ldeg_n(q_0)=m-k$. Recall that either $p'=0$
	or ($\hdeg_n(p')<k$ and $\ldeg_n(p')=\ldeg_n(p)$). So, either $q_0p'=0$, or 
	($\hdeg_n(q_0p')<m$ and $\ldeg_n(q_0p')=m-k+\ldeg_n(p)=m-\wdeg_n(p)$).
	
	Recall also that either $f'=0$, or ($\hdeg_n(f')<m$ and $\ldeg_n(f')=\ldeg_n(f)$).
	So, if $p'$ and $f'$ are not both 0, then $\hdeg_n(f'')<m$.
	
	Also, if $p'$ and $f'$ are not both 0, then 
	$\ldeg_n(f'')\ge \min(m-\wdeg_n(p), \ldeg_n(f))$. We know that $m=\hdeg_n(f)\ge d+\wdeg_n(p)$, so $m-\wdeg_n(p)\ge d$. 
	Also, $\ldeg_n(f)\ge d$. So, $\ldeg_n(f'')\ge d$, and we can apply the induction hypothesis.
	
	By induction, there exist polynomials $q_1$ and $r$ such that 
	$f''=pq_1+r$, and (either $r=0$ or ($\ldeg_n(r)\ge d$ and $\hdeg_n(r)<d+\wdeg_n(p)$)).
	
	Set $q=q_0+q_1$. Then 
	\begin{align*}pq+r &=pq_0+pq_1+r=(Yx_n^k+p')q_0+f'' \\ 
	& =Yx_n^kq_0+p'q_0+f'-q_0p'=gx_n^m+f'=f. \qedhere\end{align*}
\end{proof}

\begin{dfn}{(cf. \cite[\S15.5]{Ei})}
	Given a $n$-tuple of polynomials $\vec{q}=(q_1,\ldots, q_n)$, a $n$-tuple  of polynomials $\vec{f}=(f_1,\ldots,f_n)$ is called a {\em syzygy} of $\vec{q}$ if $\sum_{i}f_iq_i=0$. 
	
	Observe that syzygies form a submodule of a free module of rank $n$ over $R[\Lambda]$. An element of a submodule generated by \[S_{ij}=(0,\ldots,q_j,\ldots,-q_i,\ldots,0),\]
	where $q_j$ is at the position $i$ and $-q_i$ is at the position $j$, $i,j=1,\ldots, n$ and $i\neq j$, is called a {\em trivial syzygy} of $\vec{q}$.
\end{dfn}

\begin{lemma}\label{syzygylifting}
Let $\vec{q}=(q_i)$, $q_i\in \Z[\Lambda]$ and let $\bar q=(\bar q_i)$, where $\bar q_i \in \Z/d\Z[\Lambda]$ is the reduction modulo $d$,
$d\ge 2$. 

If $\vec{f}'=(f_i')$ is a trivial syzygy of $\bar q$, then there exists a trivial syzygy $\vec{f}=(f_i)$ of $\vec{q}$ such that its reduction modulo $d$ coincides
with $\vec{f'}$, i.e., a trivial syzygy can be always lifted to $\Z$.
\end{lemma}

\begin{proof} Let $\bar{S}_{ij}$ be the reduction modulo $d$ of $S_{ij}$.
We have
$\vec{f}'=\sum_{i,j}g_{ij}'\bar{S}_{ij}$ for some $g_{ij}'\in \Z/d\Z[\Lambda]$.
Let $g_{ij}$ be liftings of $g_{ij}'$ to $\Z[\Lambda]$. Set $\vec{f}=\sum_{i,j} g_{ij}S_{ij}$.
\end{proof}


\begin{dfn}
We say that a $n$-tuple of polynomials $(q_1,\ldots, q_n)$ satisfies the flatness
condition if 
$q_i \in R[\Lambda_i]$ for each $i=1,\ldots,n$, and $q_i$ is a divisor with respect to $x_i$.
\end{dfn}

\begin{lemma}\label{lem:flatn}
	If an $n$-tuple of polynomials $\vec{r}$ satisfies flatness property, then all syzygies of $\vec{r}$ are trivial.
\end{lemma}
\begin{proof}
	First, consider the case where $R$ is a domain (i.e., $R=\Z$ or $\Z/p\Z$ with $p$ prime).	We use induction on $n$. If $n=1$, then the trivial syzygy $0$ is the only syzygy. Let $\vec{f}=(f_{1},\ldots, f_{n})$ be a syzygy of $\vec{r}=(r_{1},\ldots, r_{n})$ with $n\geq 2$. By Lemma \ref{intdivision}, we can divide $f_{i}=r_{n}g_{i}+h_{i}$ with bound $d=\min\{\ldeg_{n}(f_{i})\}$, where $\ldeg_{n}(h_{i})\geq d$ and $\hdeg_{n}(h_{i})<d+\wdeg_{n}(r_{n})$ for $1\leq i\leq n$. 
	
	Since $(g_{1}r_{1},\ldots, g_{n-1}r_{n-1}, -\sum_{i=1}^{n-1}g_{i}r_{i})$ is a trivial syzygy of $\vec{r}$, it suffices to show that $(h_{1},\ldots, h_{n-1}, f_{n}+\sum_{i=1}^{n-1}g_{i}r_{i})$ is a trivial syzygy of $\vec{r}$. If $\sum_{i=1}^{n-1}h_{i}r_{i}=(f_{n}+\sum_{i=1}^{n-1}g_{i}r_{i})r_{n}$ is nonzero, then by taking $\wdeg_{n}$ of both sides, we obtain
	\[\wdeg_{n}(r_{n})>\wdeg_{n}(\textstyle\sum_{i=1}^{n-1}h_{i}r_{i})=\wdeg((f_{n}+\sum_{i=1}^{n-1}g_{i}r_{i})r_{n})\geq \wdeg(r_{n}),\]
	a contradiction. Thus, we have $f_{n}+\sum_{i=1}^{n-1}g_{i}r_{i}=0$ and it remains to show that $\vec{h}=(h_{1},\ldots, h_{n-1},0)$ is a trivial syzygy of $\vec{r}$. Let $e=d+\wdeg_{n}(r_{n})-1$. Write $h_{i}=h_{id}x_{n}^{d}+\cdots +h_{ie}x_{n}^{e}$ for all $1\leq i\leq n-1$ and $\vec{h}_{j}=(h_{1j},\ldots, h_{(n-1)j},0)$ for all $d\leq j\leq e$. Then, we have 
	$\vec{h}=\vec{h}_{d}x_{n}^{d}+\cdots +\vec{h}_{e}x_{n}^{e}$. By induction, all syzygies $\vec{h}_{j}$ are trivial, so is $\vec{h}$. 
	
	Now we consider the case $R=\Z/m\Z$. We proceed by induction on the number of prime factors in $m$. 
	If $m$ is a prime, it follows from the previous case.
	
	Write $m=pl$, $l>1$, where $p$ is a prime. Let $\vec{f}=(f_i)$ be a syzygy of $\vec{r}=(r_i)$ and let $\bar{f}=(\bar{f_{i}})$ be the corresponding syzygy of $\bar{r}=(\bar{r}_{i})$ over $\bar{R}=\Z/l\Z$ for $1\leq i\leq n$. By induction, we have 
	\begin{equation*}\tag{*}
	\bar{f}=\sum_{i, j}\bar{g}_{ij}\bar{S}_{ij} \quad\text{  for some } \bar{g}_{ij}\in \bar{R}[\Lambda].
	\end{equation*}
	Set \[
	\vec{f}'=\vec{f}-\sum_{i,j}g_{ij}S_{ij},\; \text{ where }g_{ij}\text{ is a preimage of }\bar{g}_{ij}\text{ in }R[\Lambda].
	\] 
	By (*) we have $f_{i}'=lf_{i}''$ for some $f_i''\in R[\Lambda]$. 
	Since $\vec{f}'$ is a syzygy of $\vec{r}$, we have $0=\sum_i f_i'r_i=l(\sum_i f_i'' r_i)$ in $R[\Lambda]$. Thus,
	$\vec{f}''=(f_i'')$  is a syzygy of $\vec{r}$ modulo $p$.
	By the previous case, $\vec{f}''$ is a trivial syzygy modulo $p$. 
	So $\vec{f}''=\sum_{i,j}g_{ij}'S_{ij}+p\vec{h}$ for some $n$-tuple of polynomials $\vec{h}$ and preimages $g_{ij}'$.
	Then $\vec{f}'=l\vec{f}''=\sum_{ij}lg_{ij}'S_{ij}$ is a trivial syzygy of $\vec{r}$.
	\end{proof}


\begin{dfn}\label{dfn:genflatcond}
We say that a $n$-tuple of polynomials $(q_1,\ldots, q_n)$ satisfies a \emph{generalized flatness condition} if
there exists a matrix $A\in GL_n(R[\Lambda])$ 
such that the $n$-tuple $(r_1,\ldots,r_n)=(q_1,\ldots,q_n)A$
satisfies the flatness condition. 
\end{dfn}

\begin{lemma}\label{lem:genflat}
Assume that a $n$-tuple of polynomials $\vec{q}=(q_1,\ldots, q_n)$ satisfies the generalized flatness condition.
Then all syzygies of $\vec{q}$ are trivial.
\end{lemma}

\begin{proof}Let $A$ be a matrix such that $\vec{r}=\vec{q}A$ satisfies the flatness condition.
Let $\vec{f}=(f_1,\ldots,f_n)$ be a syzygy of $\vec{q}$. Then (as a product of matrices)
\[
0=\vec{q}\cdot \vec{f}^t=\vec{q}A\cdot A^{-1}\vec{f}^t=\vec{r}\cdot \vec{g},\text{ where }\vec{g}=A^{-1}\vec{f}^t.
\]
Hence, $\vec{g}$ is a syzygy of $\vec{r}$ and $\vec{f}=A\vec{g}$.
By Lemma~\ref{lem:flatn} it suffices to prove that if $\vec{g}=S_{ij}$ is a trivial syzygy of $\vec{r}$,
then $A\vec{g}$ is a trivial syzygy of $\vec{q}$.

Let $M_{ij}$, $i\neq j$ denote a matrix where all entries are zeros except $1$ at the position $(i,j)$ and $-1$ at the position $(j,i)$.
The matrix $M_{ij}$ is skew-symmetric.
By definition, we have $S_{ij}=M_{ij}(\vec{r})^t$.
So all trivial syzygies of $\vec{r}$ are linear combinations with coefficients in $R[\Lambda]$
of $M_{ij}\vec{r}^t$.
Similarly, all trivial syzygies of $\vec{q}$ are linear combinations 
of $M_{ij}\vec{q}^t$.

Then we obtain
$A\vec{g}=AM_{ij}\vec{r}^t=AM_{ij}A^{t}\vec{q}^t$.
Finally, since the matrix $AM_{ij}A^{t}$ is skew-symmetric, it is a linear combination with coefficients in $R[\Lambda]$ of matrices $M_{i'j'}$ for various $i'$, $j'$.
\end{proof}

\section{The generators}\label{sec:gen}

Consider the weight lattice $\Lambda$ of a semisimple root system corresponding to a group $G$. 
Let $T^*$ be a group of characters of a split maximal torus $T$ of $G$.
We assume that $T^*$ is of index 2 in $\Lambda$, i.e., $\Lambda/T^*=\Z/2\Z$. 

Consider the $\Z/2\Z$-grading on $\Lambda$ given by:
a weight $\lambda \in \Lambda$ has degree $|\lambda|$ which is 
its class in the quotient $\Lambda/T^*$.
We denote by $\Lambda^{(0)}=T^*$ the subgroup of $\Lambda$ of degree $0$ and by $\Lambda^{(1)}=\Lambda\setminus T^*$
the subset of degree $1$. 

There is an induced grading on the group ring $R[\Lambda]$ so that 
$R[\Lambda]=R[\Lambda^{(0)}]\oplus R[\Lambda^{(1)}]$. 
Hence, we can uniquely express any $f \in R[\Lambda]$ as a sum of its homogeneous components, i.e., $f=f^{(0)}+f^{(1)}$.
We say that $f\in R[\Lambda]$ is homogeneous of degree $0$ or, equivalently, $\deg(f)=0$ (resp. $f$ is of degree $1$ or $\deg(f)=1$) if $f\in R[\Lambda^{(0)}]$ (resp. $f\in R[\Lambda^{(1)}]$).

Let $\{\omega_1,\ldots,\omega_n\}$ denote the set of fundamental weights (a $\Z$-basis of $\Lambda$).
Consider the orbit $W(\omega_i)$ of the fundamental weight $\omega_i$ by means of the Weyl group $W$. 
We denote by $|i|$ the degree of $\omega_i$ with respect to the grading and by $|W(\omega_i)|$ the number of elements in the orbit.
Let \[d=\gcd_{\omega_i \in \Lambda^{(1)}}(|W(\omega_i)|)=\gcd_{|i| =1}(|W(\omega_i)|).\] 

We set $R=\Z$ and we denote by bar the reduction modulo $d$, i.e., $\bar R=\Z/d\Z$.
We define \[\rho(\omega_i)=\sum_{\lambda\in W(\omega_i)}e^{\lambda}\quad\text{ and }
\rho_i=\rho(\omega_i)-|W(\omega_i)|\in R[\Lambda].
\] 
Since the Weyl group acts trivially on $\Lambda/T^*$, we have $\deg(\rho(\omega_i))=|i|$.
Reducing modulo $d$ we obtain $\deg(\bar \rho_i)=\deg(\overline{\rho(\omega_i)})=|i|$. 

We will need the following

\begin{lemma}
\label{finaldizisibilitylemma}
Assume that $(\bar\rho_1,\ldots,\bar\rho_n)$ satisfies the generalized flatness condition with respect to some basis $\{x_i\}$ of $\Lambda$. 
Assume that $f_i\in R[\Lambda]$, $i=1,\ldots, n$ are such that $\deg(\sum_i f_i\rho_i)=0$.

Then there exist polynomials $g_1,\ldots,g_n \in R[\Lambda]$ such that
$\sum_i f_i\rho_i=\sum_i g_i\rho_i$ and 
$\bar g_i^{(1-|i|)}=0$ for each $i$.
\end{lemma}

\begin{proof}
Since $\deg(\sum_i f_i\rho_i)=0$, we have $\sum_i \bar f_i^{(1-|i|)} \bar\rho_i =0$.
Hence, by Lemma~\ref{lem:genflat} the $n$-tuple $(\bar f_i^{(1-|i|)})$ is a trivial syzygy of $(\bar\rho_i)$.
By Lemma~\ref{syzygylifting} there exist a trivial syzygy $(h_i)$, $h_i\in R[\Lambda]$, of $(\rho_i)$ such that $\bar h_i=\bar f_i^{(1-|i|)}$.
Set $g_i=f_i-h_i$. 
\end{proof}

After a possible reindexing, we may assume that the first $n'$ fundamental weights
$\{\omega_1,\ldots,\omega_{n'}\}$ have degree $1$ and the remaining fundamental weights have degree $0$. 
For $1\le i\le n'$ we set
\[
d_i=\gcd_{i\le j\le n'} (s_j),\quad \text{ where }s_j=|W(\omega_j)|.
\]
So we have $d=d_1\mid d_2 \mid \ldots \mid d_{n'}=s_{n'}$.
By a presentation of the gcd, there exist integers (denoted $a_{i,j}$, $1\le i\le j\le n'$)
such that
\[
d_i=a_{i,i}s_i+a_{i,i+1}s_{i+1}+\ldots+a_{i,j}s_j+\ldots+a_{i,n'}s_{n'}.
\]
For $1\le i<n'$ we set 
\[
\tilde{\rho}_i=a_{i,i}\rho_i+a_{i,i+1}\rho_{i+1}+\ldots + a_{i,n'}\rho_{n'}.
\]
By definition, the coefficient of $\tilde\rho_i$ at $1=e^0$ is
\[
-a_{i,i}s_i-a_{i,i+1}s_{i+1}-\ldots -a_{i,n'}s_{n'}=-d_i.
\]
Set $\tilde \rho(\omega_i)=\tilde{\rho}_i+d_i\in R[\Lambda^{(1)}]$.

\begin{dfn}\label{dfn:genlist}
Fix $\lambda_0\in \Lambda^{(1)}$ and consider the following subsets of elements in $R[T^*]$:

\begin{itemize}
\item[(1)]  $\{h_{1,i}=e^{\lambda_0}(\frac{r_i}{s_i}\rho(\omega_i)-\frac{r_i}{d_{i+1}}\tilde{\rho}(\omega_{i+1}))\mid 1\le i<n',\; r_i=\lcm (s_i,d_{i+1})\}$.

\item[(2)] $\{h_{2,i}=\rho(\omega_i)\tilde{\rho}(\omega_1)-ds_i\mid 1\le i\le n'\}$,

\item[(3)] $\{h_{3,i}=\rho_i \mid n'<i\le n \}$.
\end{itemize}
\end{dfn}

\begin{remark}
The elements $h_{1,i}$ will be extensively used (see \eqref{eq:typeBgen} and \eqref{eq:typeCgen}) in the computations of the group of semi-decomposable invariants.
\end{remark}

Let $I$ be the augmentation ideal of $R[\Lambda]$, that is $I$ is the kernel of the map $R[\Lambda]\to R$ given by $e^{\lambda}\mapsto \lambda$.
Let $I^W_{sc}$ denote the ideal in $R[\Lambda]$ generated by elements from $R[\Lambda]^W\cap I$. By the Chevalley theorem 
$I_{sc}^{W}$ is generated by the elements $\rho_i$, $1\le i\le n$, i.e., any $f\in I_{sc}^{W}$ can be written as
$f=f_1\rho_1+\ldots +f_n\rho_n$ for some $f_i\in R[\Lambda]$.

Our main result is the following

\begin{theorem}\label{thm:maingen} Assume that the $n$-tuple
$(\rho_1,\ldots,\rho_n)$ satisfies the generalized flatness condition with respect to some basis $\{x_i\}$ of the weight lattice $\Lambda$. 

Then the elements $h_{k,i}$ of Definition~\ref{dfn:genlist} generate the ideal $I_{sc}^W\cap R[T^*]$ in $R[T^*]$.
\end{theorem}

\begin{proof}
Suppose that $f_1\rho_1+\ldots+f_n\rho_n\in R[T^*]$ for some $f_i\in R[\Lambda]$.
By Lemma \ref{finaldizisibilitylemma} we may assume that $d\mid f_i^{(1-|i|)}$ for each $i$.

To prove the theorem we modify $(f_1,\ldots, f_n)$ in several steps. 
At each step, we subtract a linear combination of the elements $h_{i,j}$
(with coefficients in $R[T^*]$) from $f_1\rho_1+\ldots+f_n\rho_n$ so that the new polynomials $f_1',\ldots, f_n'$ have less non-zero monomials. 
In the end they will all become $0$, so that the original $f_1\rho_1+\ldots+f_n\rho_n$ will be replaced by a linear 
combination of $h_{i,j}$.

\noindent
{Step 1.} 
By definition we have for $1\le i\le n'$
\begin{align*}
h_{2,i} &=\rho(\omega_i)\tilde{\rho}(\omega_1)-ds_i=\rho(\omega_i)(\tilde{\rho}_1+d_1)-ds_i \\
 &=\rho(\omega_i)(a_{1,1}\rho_1+a_{1,2}\rho_{2}+\ldots + a_{1,n'}\rho_{n'}+d)-ds_i \\
 &=a_{1,1}\rho(\omega_i)\rho_1+\ldots+(a_{1,i}\rho(\omega_i)+d)\rho_i +\ldots +a_{1,n'}\rho(\omega_i)\rho_{n'}.
\end{align*}
Since $|i|=1$, $d\mid f_i^{(0)}$. Consider the difference
\begin{equation}\tag{*}
f_1'\rho_1+\ldots+f_n'\rho_n=f_1\rho_1+\ldots+f_n\rho_n-\tfrac{1}{d}f_i^{(0)}h_{2,i}.
\end{equation}
Collecting the coefficients we obtain 
\begin{itemize}
\item
$f_j'=f_j$ for all $j>n'$, 
\item $f_j'^{(0)}=f_j^{(0)}$ for all $j\le n'$ and $j\neq i$, 
\item $f_i'^{(0)}=0$.
\end{itemize}
Hence, applying~(*) for each $i$, $1\le i\le n'$
we obtain new coefficients $(f_1',\ldots,f_n')$ such that 
$f_j'^{(0)}=0$ for all $j\le n'$ and $f_j'=f_j$ for all $j>n'$.
(Observe that $f_j'^{(1)}$ for $j\le n'$ does not necessarily coincides with $f_j^{(1)}$.)

\noindent
{Step 2.}  We have for $n'< i \le n$
\begin{align*}
\tilde{\rho}(\omega_1) h_{3,i}&=\tilde{\rho}(\omega_1)\rho_i=(a_{1,1}\rho_1+a_{1,2}\rho_{2}+\ldots + a_{1,n'}\rho_{n'}+d)\rho_i \\
 &=(a_{1,1}\rho_i)\rho_1+\ldots +(a_{1,n'}\rho_i)\rho_{n'}+d\rho_i.
\end{align*}
Since $|i|=0$, $d\mid f_i^{(1)}$. Consider the difference
\begin{equation}\tag{**}
f_1'\rho_1+\ldots+f_n'\rho_n=f_1\rho_1+\ldots+f_n\rho_n-\tfrac{1}{d}f_i^{(1)}\tilde{\rho}(\omega_1)h_{3,i}.
\end{equation}
Collecting the coefficients we obtain 
\begin{itemize}
\item
$f_j'=f_j$ for all $j>n'$ and $j\neq i$, 
\item $f_j'^{(0)}=f_j^{(0)}$ for all $j\le n'$.
\item $f_i'^{(1)}=0$.
\end{itemize}
Hence, applying~(**) for each $i$, $n'< i\le n$
we obtain new coefficients $(f_1',\ldots,f_n')$ such that 
$f_j'^{(1)}=0$ for all $j>n'$ and $f_j'^{(0)}=0$ for all $j\le n'$.

\noindent
{Step 3.}  As a result of step 2, we have $f_i\in R[T^*]$ for all $i>n'$. Subtracting 
\[
f_1\rho_1+\ldots +f_n\rho_n-\sum_{i>n'}f_i h_{3,i}
\]
we may assume that  $f_i=0$ for all $i>n'$.

\noindent
{Step 4.} Fix $i$, $1\le i\le n'$. If $i>1$ we assume in addition that $f_1=\ldots=f_{i-1}=0$. 
So, we have $f_i\rho_i+\ldots+f_{n'}\rho_{n'}\in R[T^*]$, where $f_j^{(0)}=0$ for all $i\le j \le n'$ by previous steps.
Hence, we can express it as
\[
f_i\rho_i+\ldots+f_{n'}\rho_{n'}=f_i^{(1)}\rho(\omega_i)+\ldots +f_{n'}^{(1)}\rho(\omega_{n'})-(s_if_i^{(1)}+\ldots+ s_{n'}f_{n'}^{(1)})\in R[T^*].
\]
Since $\deg(\rho(\omega_j))=1$ for $i\le j\le n'$, we obtain
\[
s_if_i^{(1)}=-s_{i+1}f_{i+1}^{(1)}-\ldots -s_{n'}f_{n'}^{(1)}.
\]
The right hand side of this equation is divisible by $r_i$, hence, $\tfrac{r_i}{s_i}\mid f_i^{(1)}$.

By definition, we have
\begin{align*}
h_{1,i} &=e^{\lambda_0}(\tfrac{r_i}{s_i}\rho(\omega_i)-\tfrac{r_i}{d_{i+1}}\tilde{\rho}(\omega_{i+1}))=
e^{\lambda_0}(\tfrac{r_i}{s_i}\rho_i-\tfrac{r_i}{d_{i+1}}\tilde{\rho}_{i+1})\\
&=e^{\lambda_0}(\tfrac{r_i}{s_i}\rho_i-\tfrac{r_i}{d_{i+1}}(a_{i+1,i+1}\rho_{i+1}+a_{i+1,i+2}\rho_{i+2}+\ldots + a_{i+1,n'}\rho_{n'} ))
\end{align*}
Consider the difference
\begin{equation}\tag{***}
f_i'\rho_i+\ldots+f_{n'}'\rho_{n'}=f_i\rho_i+\ldots+f_{n'}\rho_{n'}-\tfrac{s_i}{r_i}f_i^{(1)}e^{-\lambda_0}h_{1,i}.
\end{equation}
Collecting the coefficients we obtain $f_i'=0$ while keeping $f_j'^{(0)}=f_j^{(0)}=0$ for all $i<j\le n'$.
Hence, applying~(***) inductively starting with $i=1$, we obtain that $f_i=0$ for all $1\le i\le n'$. 
\end{proof}

\section{The generalized flatness condition}

In the present section we prove that the $n$-tuple of $W$-orbits $(\rho_1,\ldots,\rho_n)$ in $R[\Lambda]$
satisfies the generalized flatness condition when $\Lambda$ is a weight lattice for a semi-simple root system of type ${\rm A}$ and ${\rm C}$.
Observe that it is enough to prove the generalized flatness condition for each simple component.

\subsection
{Type ${\rm A}$}
Let $\widetilde\Lambda=\Z^{n+1}$ with a standard basis $e_1,\ldots,e_{n+1}$.
The weight lattice of type ${\rm A}$ is then given by $\Lambda=\widetilde\Lambda/(e_1+\ldots+e_{n+1})$. 
We denote the class of $e_i$ in $\Lambda$ by $\bar e_i$.
The basis of $\Lambda$ is given by the fundamental weights $\omega_i=\bar e_1+\ldots +\bar e_i$, $i=1,\ldots, n$.
The Weyl group (the symmetric group $S_{n+1}$) acts by permutations of $\{e_1,\ldots,e_{n+1}\}$. 
Let $x_i=e^{\omega_i}$ in $\Z[\Lambda]$ and let $y_i=e^{e_i}$ in $\Z[\tilde \Lambda]$.

Consider the induced map $\phi\colon\Z[\tilde\Lambda] \to \Z[\Lambda]$ given by $\phi(y_1)=x_1$, $\phi(y_i)=e^{\bar e_i}=x_ix_{i-1}^{-1}$, $1<i\le n$ and $\phi(y_{n+1})=x_n^{-1}$. The image of the elementary symmetric function $\sigma_i=\sigma_i(y_1,\ldots,y_{n+1})$
gives the $W$-orbit $\rho(\omega_i)$.

Let $g_i$ be (the complete sum symmetric function) the sum of all monomials of total degree $i$ in variables $y_1,\ldots,y_{n+2-i}$. We have the following analogue of the Newton relation (see \cite[Relation~(2)]{MS}) for $i>0$
\begin{equation}\label{eq:newton}
\sum_{j=0}^{i}(-1)^j g_{j}(y_1,\ldots,y_{n+2-j})\sigma_{i-j}(y_1,\ldots,y_{n+1-j})=0\quad (\text{here }g_0=\sigma_0=1)
\end{equation}
which implies that the ideal $I_\sigma=(\sigma_1,\ldots,\sigma_{n+1})$ in $\Z[y_1,\ldots,y_{n+1}]$ 
coincides with the ideal $I_g=(g_1,\ldots, g_{n+1})$.
Consider  the involution $\tau$ of $\Z[y_1,\ldots,y_{n+1}]$
given by $y_i\mapsto 1-y_i$. We get
\[
(\sigma_1-s_1,\ldots, \sigma_{n+1}-s_{n+1})=\tau(I_\sigma)=\tau(I_g)=(\tilde{g}_1,\ldots,\tilde{g}_{n+1}),
\]
where $s_i=\sigma_i(1,\ldots, 1)=|W(\omega_i)|$ and $\tilde{g}_i$ is the (non-homogeneous) polynomial in variables $y_1,\ldots,y_{n+2-i}$ of degree $i$ such that its coefficient at $y_{n+2-i}^i$ is $\pm 1$. Taking its images in $\Z[\Lambda]$ we obtain
\[
(\rho_1,\ldots,\rho_n)=\phi(\tau(I_\sigma))=\phi(\tau(I_g))=(r_1,\ldots,r_{n+1}),
\]
where $\rho_i=\rho(\omega_i)-s_i$ and $r_i=\phi(\tilde{g}_{n+2-i})$.
We claim that $r_{n+1}$ can be written as a linear combination of $r_1,\ldots,r_n$. Indeed,
taking the sum of relations~\eqref{eq:newton} for all $i$ we obtain
\[
1=\sum_{i=0}^{n+1} (g_i\prod_{j=1}^{n+1-i} (1-y_j)).
\]
After applying $\tau$ we obtain 
$1=\sum_{i=0}^{n+1} \tilde{g}_i y_1\ldots y_{n+1-i}$.
Since $y_1\ldots y_{n+1}=1$, after taking its image in $\Z[\Lambda]$ we obtain the desired linear combination.

\subsection{Type C}
Consider the weight lattice $\Lambda$ of type ${\rm C}$. It is generated by the standard vectors $\{e_1,\ldots,e_n\}$
with fundamental weights $\omega_i=e_1+\ldots+e_i$. The Weyl group $W$ acts on the standard vectors by permutations and changing signs.
Consider the embedding $\phi\colon \Z[y_1,\ldots,y_n]\hookrightarrow \Z[\Lambda]$ given by $\phi(y_i)=e^{e_i}+e^{-e_i}$.
The image of the elementary symmetric function $\sigma_i=\sigma_i(y_1,\ldots,y_n)$ gives the $W$-orbit $\rho(\omega_i)$.

As in type ${\rm A}$ let $g_i$ be the sum of all monomials of total degree $i$ in variables $y_1,\ldots,y_{n+1-i}$.
Then $(\sigma_1,\ldots,\sigma_n)=(g_1,\ldots,g_n)$ as ideals in $\Z[y_1,\ldots,y_n]$. Applying the involution $\tau$ we obtain
$(\rho_1,\ldots,\rho_n)=(r_1,\ldots,r_n)$, where $\rho_i=\rho(\omega_i)-s_i$, $r_i=\phi(\tau(g_{n+1-i})$ and $s_i=|W(\omega_i)|$.
The $n$-tuple $(r_1,\ldots,r_n)$ satisfies flatness condition and there is an invertible transformation matrix between $\rho_i$ and $r_i$.

\section{Characters and invariants}

In the present section we introduce some notation and recall basic definitions for the group of characters, characteristic classes and invariants
which will be used in the subsequent sections. We follow \cite{Mer}, \cite{MNZ} and \cite{Bae}.

\subsection{Characters.}\label{par:char}
Let $H$ and $H'$ be simply connected simple split groups of the same Dynkin type $\mathcal{D}$ over a field $F$. Assume that there is a central diagonal subgroup $\mu_k$ in $H\times H'$.
The quotient $G=(H\times H')/\mu_k$ will be called a {\em group of index k} of type $\mathcal{D}$.

We denote by $T_{sc}$ the split maximal torus of $H\times H'$, by $T$ the split maximal torus of $G$ and by $T_{ad}$ the split maximal torus
of the product of the adjoint forms $H_{ad}\times H_{ad}'$.
Then there is an exact sequence for the groups of characters
\[
0\to T^*/T^*_{ad} \to T^*_{sc}/T^*_{ad} \to \Z/k\Z \to 0
\]
which can be used to describe $T^*$. Indeed, the quotient $T^*_{sc}/T^*_{ad}$ is the group of characters of the center $Z(G)$ 
and the map $T^*_{sc}/T^*_{ad} \to \Z/k\Z$ is induced by the diagonal embedding $\mu_k\to Z(G)$. 
Moreover, $T^*_{sc}/T^*_{ad}$ is the product of groups of characters of the centers of $H$ and $H'$, hence, 
\[
T^*_{sc}/T^*_{ad}=\Lambda_w/\Lambda_r \oplus \Lambda_w'/\Lambda_r' \to \Z/k\Z,
\] 
is given by taking the sum,
where $\Lambda_w$ (resp. $\Lambda_w'$) is the weight lattice and $\Lambda_r$ (resp. $\Lambda_r'$) is the root lattice of $H$ (resp. of $H'$).

\subsection{Invariant forms.}\label{par:invar}
Let $W=W_H\times W_{H'}$ be the Weyl group of $H\times H'$.
It acts on $T_{sc}^*=\Lambda_w \oplus \Lambda_w'$ as the Weyl groups of $H$ and $H'$.
Consider the group of $W$-invariant quadratic forms. It is a direct sum of cyclic groups 
\[
S^2(T_{sc}^*)^W=S^2(\Lambda_w)^{W_H}\oplus S^2(\Lambda_w')^{W_{H'}}=\Z q\oplus \Z q', \]
where $q$ and $q'$ are generators given by normalized Killing forms.
So any form $\phi \in S^2(T_{sc}^*)^W$ can be written uniquely as $\phi=dq+d'q'$, $d,d'\in \Z$.
The list of Killing forms for all types can be found in \cite[\S4]{Mer}.

Let 
$\{\omega_i\}$ and $\{\omega_i'\}$ denote the fundamental weights of $H$ and $H'$, i.e., the $\Z$-bases of $\Lambda_w$ and $\Lambda_w'$.
Choose a $\Z$-basis $\{x_i\}$ of $T^*$. Expressing each $\omega_i$ and $\omega_i'$ in terms of $x_j$'s  and substituting into $\phi$
allows to explicitly describe the subgroup \[Q(G)=S^2(T^*)^W=S^2(T^*)\cap S^2(T^*_{sc})^W.\]

\subsection{Characteristic map} Consider the group ring $\Z[T^*_{sc}]$ that is the representation ring $R(T_{sc})$ of $T_{sc}$.
Each element of $\Z[T^*_{sc}]$ can be written as a finite linear combination $\sum_i a_i e^{\lambda_i}$, $a_i\in \Z$, $\lambda_i\in T^*_{sc}$.
Fix a basis of $T_{sc}^*$ consisting of fundamental weights $\{\omega_i\}$ and $\{\omega_j'\}$.
Following \cite[\S3c]{Mer} and \cite{GaZa} define the map
\[
c_2\colon \Z[T^*_{sc}] \to S^{\le 2}(T^*_{sc}) \to S^2(T_{sc}^*)
\]
by sending $1\mapsto 1$, $e^{-\omega_i}\mapsto 1-\omega_i$ and $e^{\omega_i}\mapsto 1+\omega_i+\omega_i^2$ (resp. for $\omega_j'$) and then taking the degree 2 homogeneous component.

Let $I_{sc}$ denote the augmentation ideal in $\Z[T^*_{sc}]$, i.e., the kernel
of the trace map $\Z[T^*_{sc}]\to \Z$, $e^{\lambda}\mapsto 1$. Then the image $c_2(I_{sc}^3)=0$, so $c_2$ can be restricted to $I_{sc}^2$\cite{GaZa}.

Observe that the filtration by powers of the ideal $I_{sc}$ can be viewed as a $\gamma$-filtration on $K_0(BT)$; its image
in $K_0(G/B)$ via the characteristic map $c\colon \Z[T_{sc}^*] \to K_0(G/B)$ gives the Grothendieck $\gamma$-filtration on $K_0(G/B)$ (e.g., see \cite{Za}).

\subsection{Invariants}
Given $\lambda\in T^*$ denote by $\rho(\lambda)=\sum_{\chi\in W(\lambda)} e^\chi$, where $W(\lambda)$ is the $W$-orbit of $\lambda$.
If restricted to invariants, the map $c_2$ defines a group homomorphism \[
c_2\colon \Z[T^*]^W \to S^2(T^*)^W\] with image
generated by forms $c_2(\rho(\lambda))=-\tfrac{1}{2}\sum_{\chi \in W(\lambda)}\chi^2$ 
for all $\lambda\in T^*$ \cite[\S3c]{Mer}.
It was shown in \cite{Mer} that the image $c_2(\Z[T^*]^W)$ can be identified with the group of degree 3 decomposable invariants $\Dec(G)$ and
the quotient $Q(G)/\Dec(G)$ with the group of indecomposable invariants denoted by $\Inv(G)$.
By definition for any two semisimple groups $G_1$, $G_2$ we have 
\begin{equation}\label{eq:product}
Q(G_1\times G_2)=Q(G_1)\times Q(G_2)\text{ and }\Dec(G_1\times G_2)=\Dec(G_1)\times \Dec(G_2).
\end{equation}

Let $I_{sc}^W$ denote an ideal in $\Z[T_{sc}^*]$ generated by $W$-invariants from the augmentation ideal $I_{sc}$, namely, $I_{sc}^W=(I_{sc}\cap \Z[T_{sc}^*]^W)$.
The main result of \cite{MNZ} says that the image $c_2(\Z[T^*]\cap I_{sc}^W)$ in $S^2(T^*)^W$ coincides with the subgroup of semi-decomposable invariants $\Sdec(G)$ and that $\Dec(H)=\Sdec(H)$ if $H$ is a simple group.
Observe that we have 
\begin{equation}\label{eq:inclus}
\Dec(G)\subseteq \Sdec(G) \subseteq Q(G).
\end{equation}
We denote by $\Sinv(G)$ the quotient $\Sdec(G)/\Dec(G)$.

\section{Type ${\rm A}$}
In the present section we consider semisimple groups of type ${\rm A}$.
The following lemma gives a simple geometric proof for the coincidence between the normalized invariants and semi-decomposable invariants (c.f. \cite{Bae}):

\begin{lemma}\label{firstlemmaA}
	Let $G=(\gSL_{m}\times \gSL_{n})/\gmu$, $m, n, k\geq 2$, where $\gmu\simeq \gmu_{k}$ is a diagonal $(central)$ subgroup. Then, $Q(G)=\Sdec(G)$, i.e., each degree $3$ indecomposable invariant of $G$ is semi-decomposable.
\end{lemma}

\begin{proof}
	We follow arguments used in \cite[\S5]{Bae}. Assume \[
	\gmu=\{(\lambda, \lambda')\in \gmu_{m}\times \gmu_{n}\mid \lambda^{k}=\lambda'^{k}=1,\, \lambda=\lambda'\}\]
	is the diagonal subgroup. The corresponding versal flag variety $X$ over the function field $F'$ of the classifying space of $G$ can be replaced by the product of Severi-Brauer varieties $\SB(A)\times \SB(A')$, where $A$ and $A'$ are central simple $F'$-algebras of degree $m$ and $n$, respectively such that 
	$k[A]=k[A']=0$ and $[A]=[A']$ in the Brauer group $\Br(F')$. Since the subgroup in $\Br(F')$ generated by $A$ and $A'$
	is generated by some division algebra $B$ such that the index and the exponent of $B$ are all equal to $k$, we obtain $\CH^{2}(X)_{\torsion}=\CH^{2}(\SB(B))_{\torsion}$. By \cite[Cor.4]{Kar95} we have $\CH^{2}(\SB(B))_{\torsion}=0$, thus by the main theorem of \cite{MNZ} $Q(G)$ coincides with $\Sdec(G)$.\end{proof}

\begin{remark}\label{firstArem}
	The same proof works for an arbitrary product $G=(\prod_{i=1}^{m}\gSL_{n_{i}})/\gmu_{k}$ with a diagonal subgroup $\gmu_{k}$, i.e., $Q(G)=\Sdec(G)$. Indeed, using arguments in \cite{Bae} one can compute the quotient \[\Inv(G)=Q(G)/\Sdec(G)=\CH^{2}(X)_{\torsion}\] for an arbitrary semisimple group $G$ of type ${\rm A}$. For instance, observe that the same arguments in the proof work if we replace the diagonal subgroup $\gmu\simeq \gmu_{2}$ by a central subgroup $\gmu_{2}\times 1\subseteq \gmu_{m}\times \gmu_{n}$ or $1\times \gmu_{2}\subseteq \gmu_{m}\times \gmu_{n}$.
\end{remark}

The following proposition deals with groups of $p$-primary index for any prime $p$, which in turn computes the $p$-primary component of $\Inv((\gSL_{m}\times \gSL_{n})/\gmu_{k})$ for any diagonal (central) subgroup $\gmu_{k}$.

\begin{proposition}\label{lem:primaryindexA}
	Let $G=(\gSL_{m}\times \gSL_{n})/\gmu$, $m, n\geq 2$, where $\gmu\simeq\gmu_{k}$ is a $p$-primary diagonal $(central)$ subgroup. Then, 
	{\small \[Q(G)=\{dq+d'q'\mid  (\tfrac{k-1}{k})(md+nd')\equiv 0 \bmod 2k\}\]}
	and 
	{\small \[\Dec(G) \simeq
	\begin{cases}
	k\Z q \oplus k\Z q' & \text{ if } p\neq 2 \text{ or } p=2, \min\{v_{2}(m), v_{2}(n)\}>v_{2}(k) , \\
	k\Z (q-q')\oplus k\Z (q+q')  & \text{ if }p=2, v_{2}(m)=v_{2}(n)=v_{2}(k), \\
	k\Z q\oplus 2k\Z q' & \text{ if } p=2, v_{2}(m)>v_{2}(k)=v_{2}(n),
	\end{cases}\]}
	where $q$ $(resp. \,\, q')$ is the normalized Killing form of $\gSL_{m}$ $(resp.\,\, \gSL_{n})$.
\end{proposition}
\begin{proof}
	Let $G=(\gSL_{m}\times \gSL_{n})/\gmu_k$,  $m, n, k\geq 2$, $k\,|\,\gcd(m,n)$, where $\gmu_{k}$ is a diagonal subgroup. Then, by~\ref{par:char} the character group of the split maximal torus $T$ of $G$ is given by
	{\small \begin{equation}\label{TcharA}
	T^{*}=\{\sum_{i=1}^{m-1}a_{i}\omega_{i}+\sum_{j=1}^{n-1}a_{j}'\omega_{j}' \mid \sum_{r=0}^{\tfrac{m-k}{k}}\sum_{i=1}^{k-1} ia_{i+rk}+ \sum_{s=0}^{\tfrac{n-k}{k}}\sum_{i=1}^{k-1} ia'_{i+sk}\equiv 0 \bmod k \}.
	\end{equation}}
	
	Following \ref{par:invar} the group of $W$-invariant quadratic forms $S^{2}(T^*_{sc})^{W}$ is generated by the normalized Killing forms 
	{\small \[
	q=\sum_{i=1}^{m-1}\omega_{i}^{2}-\sum_{i=1}^{m-2}\omega_{i}\omega_{i+1}\text{ and } 
	q'=\sum_{j=1}^{n-1}\omega_{j}'^{2}-\sum_{j=1}^{n-2}\omega_{j}'\omega_{j+1}'.
	\]}

Consider the $\Z$-basis $\{x_{1},\ldots, x_{m-1},x_{1}',\ldots, x_{n-1}'\}$ of the character group $T^{*}$ given by
{\small \[
x_{i+rk}=\omega_{i+rk}+i\omega_{n-1}',\, x_{k+rk}=\omega_{k+rk}\text{ and } 
x_{i+sk}'=\omega_{i+sk}'+i\omega_{n-1}',\, x_{k+sk}'=\omega_{k+sk}',
\]} 
where $1\leq i\leq k-1$, $0\leq r\leq \tfrac{m-k}{k}$, $0\leq s\leq \tfrac{n-k}{k}$ (for convenience, we set $x_{j}=x_{j}'=0$ for any $j>n-1$ or $j>m-1$). 
	In this basis a form $\phi \in S^{2}(T^*_{sc})^{W}$ can be written as
	\[
	\phi=(\tfrac{md(k-1)+nd'(k-1)}{2k^{2}})x_{n-1}'^{2}+\psi,\quad d,d'\in \Z
	\]
	where $\psi$ is a quadratic form with integer coefficients. Hence,  we obtain
	\begin{equation}\label{QGA}
		Q(G)=\{dq+d'q'\mid  (\tfrac{k-1}{k})(md+nd')\equiv 0 \bmod 2k\}.
	\end{equation}
	
	From now on we assume that $k$ is $p$-primary. We claim that $\Dec(G)\subseteq k\Z q\oplus k\Z q'$. To show this we extend the arguments in \cite[p.136]{GMS}. We use the standard presentation of the root system of type ${\rm A }$, namely, that $\Lambda_w$ (resp. $\Lambda_w'$) consists of vectors in the standard basis $\{e_1,..,e_{m}\}$ (resp. $\{e_1',..,e_{n}'\}$) whose sum of coordinates is zero.

	Choose a character $\chi\in T^{*}$. Assume $\chi$ has $l$ (resp. $l'$) distinct coordinates in some order $b_{1}>\cdots>b_{l}$ (resp. $b_{1}'>\ldots>b_{l'}'$) which repeat $r_1,..,r_l$ (resp. $s_1,..,s_{l'}$) times with respect to the basis $\{e_{i}\}$ (resp. $\{e_j'\}$). Then, for the orbit $\rho(\chi)$ of $\chi$ under the action of $W$ we obtain
	\begin{equation}\label{ctworhochi}
	c_{2}(\rho(\chi))=\tfrac{n!}{s_{1}!\cdots s_{l'}!}\cdot [r,b]\cdot q+\tfrac{m!}{r_{1}!\cdots r_{l}!}\cdot [s,b']\cdot q', \text{ where}
	\end{equation}
	{\small \[
	[r,b]=\tfrac{(m-2)!}{r_{1}!\cdots r_{l}!} \big(m(\sum_{i=1}^l r_{i}b_{i}^{2})-(\sum_{i=1}^l r_{i}b_{i})^{2}\big), [s,b']=\tfrac{(n-2)!}{s_{1}!\cdots s_{l}!} \big(n(\sum_{j=1}^{l'} s_{j}b_{j}'^{2})-(\sum_{j=1}^{l'} s_{j}b_{j}')^{2}\big).
	\]}
	Observe that by (\ref{TcharA}) we have $\sum r_{i}b_{i}+\sum s_{j}b_{j}'\equiv 0 \bmod k$. Let $c=\min\{v_{p}(r_{i})\}$ and $d=\min\{v_{p}(s_{j})\}$. If $v_{p}(k)\leq d$, then $v_{p}(k)\leq v_{p}(\sum r_{i}b_{i})$. So we obtain $v_{p}(k)\leq v_{p}(\gcd(m, \sum r_{i}b_{i}))$. Hence, by \cite[p.137, Lemma 11.4]{GMS} the coefficient $[r, b]$ is divisible by $k$. Similarly, if $v_{2}(k)\leq c$, then $[s, b']$ is divisible by $k$. Hence, we may assume that $c, d\leq v_{p}(k)$. By \cite[p.137, Lemma 11.3]{GMS}, we have
	\[
	v_{p}\big(\tfrac{n!}{s_{1}!\cdots s_{l'}!}\big)\geq v_{p}(k)-d \geq 0
	\text{ and } v_{p}\big(\tfrac{m!}{r_{1}!\cdots r_{l}!}\big)\geq v_{p}(k)-c\geq 0,
	\]
	which implies that 
	\[v_{p}(\gcd(m, \sum r_{i}b_{i}))\geq d \text{ and } v_{p}(\gcd(n, \sum s_{j}b_{j}'))\geq c.\]
	Therefore, again by \cite[p.137, Lemma 11.4]{GMS} we see that each coefficient of $q$ and $q'$ in (\ref{ctworhochi}) is divisible by $k$, which proves the claim.

	Finally, we compute the group $\Dec(G)$ case by case. As 
	\begin{equation}\label{DA}
	c_{2}(\rho(k\omega_{1}))=-k^{2}q,\, c_{2}(\rho(2\omega_{1}-\omega_{2}))=-2mq
	\end{equation}
	and similarly, $c_{2}(\rho(k\omega_{1}'))=-k^{2}q'$, $c_{2}(\rho(2\omega_{1}'-\omega_{2}'))=-2mq'$, we have 
	\begin{equation}\label{DA2}
	\gcd(2,p)kq\in \Dec(G) \text{ if }v_{p}(m)=v_{p}(k)
	\end{equation}
	and similarly, $\gcd(2,p)kq'\in \Dec(G)$ if $v_{p}(n)=v_{p}(k)$. Moreover, we get
	\begin{equation}\label{DA3}
	c_{2}(\rho(\omega_{k}))=(kk')q \text{ with }\gcd(k',p)=1 \text{ if } v_{p}(m)>v_{p}(k)
	\end{equation}
	and similarly, $c_{2}(\rho(\omega_{k}))=(k k'')q$ with $\gcd(k'',p)=1$ if $v_{p}(n)>v_{p}(k)$ (see also \cite[Thm.4.1]{BR}). Thus if $p\neq 2$, then by (\ref{DA}) and (\ref{DA3}) we obtain 
	\[\gcd(k^{2}, kk')q=kq\in \Dec(G),\quad \gcd(k^{2}, kk'')q=kq'\in \Dec(G).\] 
	Therefore, by (\ref{DA2}) and the claim above, $\Dec(G)=k\Z q\oplus k\Z q'$ if $p\neq 2$.

	Now assume that $p=2$. If $v_{2}(m)>v_{2}(k)$ and $v_{2}(n)>v_{2}(k)$, then by (\ref{DA}), (\ref{DA3}) and the claim above, we have $\Dec(G)=k\Z q\oplus k\Z q'$. 
	
	If $v_{2}(m)=v_{2}(n)=v_{2}(k)$, then $c_{2}(\rho(\omega_{k/2}+\omega'_{k/2}))\equiv -kq-kq' \bmod 2k$. Hence, by (\ref{DA2}), $k(q-q')$, $k(q+q')\in \Dec(G)$. Since $\Dec(\gSL_{m}/\gmu_{k})=2kq$ if $v_{2}(m)=v_{2}(k)$, and $\Dec(\gSL_{n}/\gmu_{k})=2kq'$ if $v_{2}(n)=v_{2}(k)$ (\cite[Thm.4.1]{BR}), it follows from the claim above that $\Dec(G)=k\Z (q-q')\oplus k\Z (q+q')$. 
	
	Similarly, if $v_{2}(m)>v_{2}(k)=v_{2}(n)$, then by (\ref{DA}) and (\ref{DA3}) we have $kq\in \Dec(G)$. Since $\Dec(\gSL_{n}/\gmu_{k})=2kq'$, we get $kq', k(q\pm q')\not\in \Dec(G)$. Therefore, by the claim we obtain $\Dec(G)=k\Z q\oplus 2k\Z q'$.\end{proof}

\begin{remark}
This proposition generalizes \cite[Theorem 4.4]{Mer} and \cite[Theorem 4.1]{BR} for split simple groups of type $A$. 

For instance, in order to obtain $Q(\gPGL_{m})$ we simply set $d'=0$ and $k=m$ in (\ref{QGA}), then we get $Q(\gPGL_{m})=2\gcd(2,m)q$ \cite[Theorem 4.4]{Mer}. Similarly, we can obtain $\Dec(\gPGL_{m})$ in the same way. 

In order to compute the indecomposable groups for $G=\gSL_{2m}/\gmu_{2}$, we set $d'=0$, $k=2$. Then it follows by Proposition \ref{lem:primaryindexA} that $Q(G)=\{dq\,|\, md\equiv 0 \bmod 4\}$ and
{\small \[\Dec(G)=\begin{cases}
2q & \text{ if }v_{2}(m)>0,\\
4q & \text{ if }v_{2}(m)=0.
\end{cases} 
\]}
Hence, we obtain \cite[Theorem 4.1]{BR}, that is	
{\small \[
\Inv(G)\simeq 
\begin{cases}
(\Z/2\Z) q & \text{ if }v_{2}(m)\geq 2, \\
0 & \text{ otherwise.}
\end{cases}
\]}
Finally observe that together with Lemma~\ref{firstlemmaA}, properties~\eqref{eq:product} and \eqref{eq:inclus} it computes the group $\Sinv(G')$, where 
		$G'=(\gSL_{2m}/\gmu_2)\times \gSL_{2n}$, $n,m\ge 1$.
\end{remark}

The following corollary generalizes \cite[Example 3.1]{MNZ} 
to groups of type ${\rm A}$ (see also \cite{Bae}). Similarly, by using Lemma \ref{firstlemmaA} and Proposition \ref{lem:primaryindexA}, one can compute both groups $\Inv(G)$ and $\Sinv(G)$ for any $p$-primary diagonal subgroup~$\gmu_{k}$. 

\begin{corollary}\label{cor:typeA}
	Let $G=(\prod_{i=1}^{m}\gSL_{2n_{i}})/\gmu_2$, $ n_{i}\geq 1$, where $\gmu_{2}$ is the diagonal subgroup. Then
	{\small\[
	\Sinv(G)=\Inv(G)=\begin{cases}(\Z/2\Z)^{\oplus m} & \text{ if }\, \forall n_{i}\equiv 0\bmod 4,\\
	(\Z/2\Z)^{\oplus m-1} & \text{ otherwise.} \end{cases}
	\]}
\end{corollary}
\begin{proof}
Let $G=(\gSL_{2m}\times \gSL_{2n})/\gmu_2$.	Then, it follows from Proposition \ref{lem:primaryindexA} that 
	{\small \[
		\Inv(G) \simeq
		\begin{cases}
		(\Z/2\Z) q \oplus (\Z/2\Z) q' & \text{ if }m\equiv n\equiv 0\bmod 4, \\
		(\Z/2\Z) q & \text{ if }m\equiv 0,  n\equiv 2 \bmod 4 \text{ or }m\equiv 0 \bmod 4, n \text{ is odd}, \\
		(\Z/2\Z) (q-q') & \text{ if }m\equiv n\equiv 2 \bmod 4\text{ or both }m, n \text{ are odd},\\
		(\Z/2\Z)(q-2q') & \text{ if } m\equiv 2 \bmod 4, n \text{ is odd}.
		\end{cases}
		\]}
Hence, the result follows by Lemma \ref{firstlemmaA}. Applying the same arguments for three and more groups completes the proof.
\end{proof}

In the following we show that the both indecomposable group $\Inv(G)$ and the semi-decomposable group $\Sinv(G)$ can have an arbitrary order (c.f. \cite{Bae}). In particular, the order can be arbitrarily large.

\begin{corollary}\label{cor:abelian} For an arbitrary integer $k\ge 2$ there exists a semisimple group $G$ of type $\rm A$ such that {\small \[
|\Inv(G)|=|\Sinv(G)|=k.\]} 
Moreover, for any homocyclic $p$-group $C$ there exists a semisimple group $H$ of type ${\rm A}$ such that $\Inv(H)=\Sinv(H)=C$.
\end{corollary}
\begin{proof}
Let $p^{r}$ ($r\geq 1$) be a prime factor of $k$ and let $n=\gcd(2,p)p^{r}$. We denote $G[p^{r}]=(\gSL_{n}\times \gSL_{n})/\gmu_{p^{r}}$, where $\gmu_{p^{r}}$ is the diagonal subgroup. By Proposition~\ref{lem:primaryindexA} we have
\[\Inv(G[p^{r}])=\Sinv(G[p^{r}])=\Z/p^{r}\Z.\]
Set $G=(\gSL_{\gcd(2,k)k}\times \gSL_{\gcd(2,k)k})/\gmu_{k}$. Then, the same argument as in \cite[\S 3b]{MNZ} shows that $\Inv(G[p^{r}])$ is a $p$-primary component of $\Inv(G)$ and the first statement follows from Lemma~\ref{firstlemmaA}.

Let $C=(\Z/p^{r}\Z)^{\oplus m}$ be a homocyclic $p$-group of rank $m$ for some prime $p$. It suffices to consider the case $m\geq 2$. Let $H=(\gSL_{n})^{m}/\gmu_{p^{r}}$, where $n=\gcd(2,p)p^{2r}$ and $\gmu_{p^{r}}$ is the diagonal subgroup. Then, the arguments used in the proof of Proposition~\ref{lem:primaryindexA} yield
{\small \[Q(G)=\{\sum_{i=1}^{m}d_{i}q_{i}\mid (n/p^{r})\sum_{i=1}^{m}d_{i}\equiv 0 \bmod \gcd(2,p)p^{r}     \},      \]}
where $q_{i}$ is the corresponding normalized Killing form of $\gSL_{n}$. Similarly, we have $\Dec(G)=\bigoplus_{i=1}^{m}p^{r}\Z q_{i}$. Then the second statement follows by Remark~\ref{firstArem}.
\end{proof}

\section{Type ${\rm B}$}

In the present section we show that any semi-decomposable invariant of semisimple groups of type $B$ is decomposable, except in the case of a product of groups of type $B_{2}=C_{2}$ modulo the diagonal subgroup $\gmu_{2}$. We first consider the index $2$ case.
 
\begin{proposition}\label{propB}
	Let  $G=(\gSpin_{2m+1}\times \gSpin_{2n+1})/\gmu_2$, $m, n\geq 2$, where $\gmu_{2}$ is the diagonal subgroup. Then, we have $\Inv(G)=\Z/2\Z$ and 
	{\small \[\Sinv(G)=\begin{cases}\Z/2\Z & \text{ if } m=n=2,\\
	0 & \text{ otherwise,} \end{cases}\]}
	i.e., each semi-decomposable invariant is decomposable unless $m=n=2$.
\end{proposition}
\begin{proof} Following \ref{par:char}
	the character group of the split maximal torus $T$ of $G$ is given by
	{\small \[
	T^{*}=\{\sum_{i=1}^{m}a_{i}\omega_{i}+\sum_{i=1}^{n}a_{i}'\omega_{i}' \,|\, a_{m}\equiv a'_{n} \bmod 2 \}.
	\]}
	Following \ref{par:invar} the group $S^{2}(T_{sc}^*)^{W}$ is generated by the normalized Killing forms
	{\small \[q=\sum_{i=1}^{m-1}(\omega_{i}^{2}-\omega_{i}\omega_{i+1})+2\omega_{m}^{2}-\omega_{m-1}\omega_{m},  q'=\sum_{i=1}^{n-1}(\omega_{i}'^{2}-\omega_{i}'\omega_{i+1}')+2\omega_{m}'^{2}-\omega_{m-1}'\omega_{m}'.\]}
	
	Choose a $\Z$-basis $\{\omega_{1},\ldots, \omega_{m-1}, \omega_{1}',\ldots, \omega_{m-1}, v_{1}, v_{1}'\}$ of $T^{*}$ where
	$v_{1}=\omega_{m}+\omega_{n}'$ and $v_{1}'=\omega_{m}-\omega_{n}'$.
	For any $\phi\in S^{2}(T^*)^{W}$ there exist $d, d'\in \Z$ such that $\phi=dq+d'q'$, thus in this basis we have  \[
	\phi=\tfrac{1}{2}(d+d')(v_{1}^{2}+v_{1}'^{2})+\psi\] for some quadratic form $\psi$ with integer coefficients. Hence, we obtain 
	\begin{equation}\label{QGB}
	Q(G)=\{dq+d'q'\,|\,  d+d'\equiv 0 \bmod 2\}=\Z(q-q')\oplus \Z(q+q').
	\end{equation}
	
	We claim that $\Dec(G)=2\Z q\oplus  2\Z q'$. The result for the group of indecomposable invariants then follows immediately.
	Indeed,
	since $q=\tfrac{1}{2}(\sum_{i=1}^{m} e_{i}^{2})$ and $q'=\tfrac{1}{2}(\sum_{j=1}^{n} e_{j}'^{2})$ in terms of the standard basis of $T_{sc}^*=\Z^{m}\oplus \Z^{n}$, we conclude that $c_{2}(\rho(\omega_1))=2q$ and $c_{2}(\rho(\omega_{1}'))=2q'$ are contained in $\Dec(G)$. On the other hand, as $\Dec(G)$ is generated by $c_{2}(\rho(\lambda))$ for all $\lambda\in T^{*}$ and the Weyl group of $G$ contains normal subgroups $(\Z/2\Z)^m$ and $(\Z/2\Z)^n$ generated by sign switching, we see that the coefficient at each $e_{i}$ in the expansion of $c_{2}(\rho(\lambda))$ is divisible by 2 (c.f. \cite[Lemma 14.2]{GMS}).

	We now compute the group $\Sdec(G)$.
	Assume that $m=n=2$. Consider an element 
	\begin{equation}\label{eq:typeBgen} y=e^{\omega_{2}}z\in \Z[T^{*}]\cap I_{sc}^{W}\text{ with }z=\bar\rho(\omega_{2})-\bar\rho(\omega_{2}'),\end{equation} where $\bar \rho(\omega_i)$ denotes the augmented orbit $\rho(\omega_i)- |W(\omega_i)|$. As $(e^{\omega_{2}}-1)z\in I_{sc}^3$, we see that $c_{2}(y)=c_{2}(z)$. Since $c_{2}(\bar\rho(\omega_{2}))=q$ and $c_{2}(\bar\rho(\omega_{2}'))=q'$, we conclude that $q-q'\in \Sdec(G)$. Therefore, $\Sinv(G)=\Z/2\Z$.

	Assume that $m, n\geq 3$. We will show that $q-q'$ which is a generator of $\Inv(G)$ does not belong to $\Sdec(G)$. Let $x\in \Z[T^*]\cap I_{sc}^W$. Similar to \cite[\S3c]{MNZ} write 
	{\small \begin{equation*}
		x=\sum_{i=1}^{m} (d_{i}+\delta_{i})\bar{\rho}(\omega_{i})+\sum_{j=1}^{n} (d_{j}'+\delta_{j}')\bar{\rho}(\omega_{j}')
	\end{equation*}}
	for some $d_{i}, d_{j}'\in \Z$ and $\delta_{i}, \delta_{j}'\in I_{sc}$. As $c_2(I_{sc}^3)=0$, we have
	{\small \[
	c_2(x)=\sum_{i=1}^{m} d_{i}c_2(\bar{\rho}(\omega_{i}))+\sum_{j=1}^{n} d_{j}'c_2(\bar{\rho}(\omega_{j}')).
	\]}
	On the other hand, we have
	\[c_{2}(\bar\rho(\omega_{i}))=2m_{i}q \text{ and } c_{2}(\bar\rho(\omega_{j}'))=2m_{j}'q'\]
	for all $1\leq i\leq m$ and $1\leq j\leq n$ and for some $m_{i}, m_{j}'\in \Z$. Hence, $c_{2}(x)\equiv 0 \bmod 2$, thus $q-q'\not\in \Sdec(G)$. Similarly, if $m=2$ and $n\geq 3$, then $c_{2}(x)\equiv (d_{1}+d_{2})q \bmod 2$, thus $q-q'\not\in \Sdec(G)$, which completes the proof.\end{proof}

The previous proposition yields the following. Combining these results we obtain both indecomposable and semi-decomposable subgroups for an arbitrary semisimple group of type ${\rm B}$.
\begin{corollary}\label{cor:typeB}
	(1) Let $G=\prod_{i=1}^{m}\gSO_{2n_i+1}$, $n_i\geq 2$. Then, {\small \[\Inv(G)=\Sinv(G)=0.\]}
		(2) Let $G=(\prod_{i=1}^{m}\gSpin_{2n_i+1})/\gmu_{2}$, $n_i\geq 2$, $m\geq 2$, where $\gmu_{2}$ is the diagonal subgroup. 
		Then, {\small \[\Inv(G)=(\Z/2\Z)^{\oplus m-1} \text{ and } \Sinv(G)=(\Z/2\Z)^{\oplus k-1}, \]} 
		where $k$ is the number of $n_{i}$'s such that $n_{i}=2$.
		
		(3) Let $G=(\prod_{i=1}^{m}\gSpin_{2n_i+1})\times (\prod_{i=1}^{m'}\gSO_{2n_{i}'+1})$, $n_{i}, n_{i}'\geq 2$. Then, {\small \[\Inv(G)=(\Z/2\Z)^{\oplus k} \text{ and } \Sinv(G)=0,  \]}
		where $k$ is the number of $n_{i}$'s such that $n_{i}\geq 3$.
\end{corollary}
\begin{proof}
(1) We set $d'=0$ in (\ref{QGB}). Then, we obtain $Q(\gSO_{2m+1})=2\Z q$ (\cite[\S 4b]{Mer}). It immediately follows from the proof of Proposition~\ref{propB} that we also have $\Dec(\gSO_{2m+1})=2\Z q$.

(2) This follows by the same argument as in Proposition~\ref{propB}.

(3) As $\Dec(\gSpin_{2n+1})=\Sdec(\gSpin_{2n+1})$ for any $n\geq 2$, the same argument as in $(1)$ shows that $\Sdec(G)=\Dec(G)$. By \cite[Theorem 13.4]{GMS}, $Q(\gSpin_{2n+1})=2\Dec(\gSpin_{2n+1})$ for any $n\geq 3$ and $\Inv(\gSpin_{5})=0$, thus the same argument as in $(1)$ proves the result for the indecomposable group.\end{proof}

\section{Type ${\rm C}$}

In the present section we compute the groups of indecomposable and semi-decomposable invariants for semisimple groups of type ${\rm C}$.
In particular, we show that for groups $G=(\prod_{i=1}^{m}\gSP_{2n_{i}})/\gmu_2$, where $m\ge 2$, $n_i\not\equiv 0\bmod 4$ for all $i=1,\ldots,m$, and $\gmu_{2}$ is the diagonal subgroup, any indecomposable invariant is semi-decomposable.

We consider the index $2$ case, which generalizes the example \cite[Example~3.1]{MNZ} (the case $n=m=1$) to groups of type ${\rm C}$.

\begin{proposition}\label{prop:typec}
	Let $G=(\gSP_{2m}\times \gSP_{2n})/\gmu_2$, $m,n\ge 1$ where $\gmu_2$ is the diagonal subgroup. Then, we have 
	{\small \[\Inv(G)=\begin{cases}\Z/2\Z\oplus \Z/2\Z & \text{ if } m\equiv n\equiv 0 \bmod 4,\\
		\Z/2\Z & \text{ otherwise,} \end{cases}\]}
	and 
	{\small \[\Sinv(G)=\begin{cases}\Z/2\Z & \text{ if } m\equiv n\equiv 0  \text{ or } m\not\equiv 0\not\equiv n \bmod 4,\\
		0 & \text{ otherwise.} \end{cases}\]}
	In particular, if both $n$ and $m$ are not divisible by 4, then each indecomposable invariant is semi-decomposable.
 \end{proposition}

\begin{proof}
Let $\{e_{1},\cdots, e_{m}, e'_{1},\cdots, e_{n}'\}$ be a standard basis of $T_{sc}^*=\Z^{m}\oplus \Z^{n}$. Then $T^*$ consists of all linear combinations of standard basis elements with even sums of coefficients.
Consider the $\Z$-basis $\{x_1,\ldots,x_m,x_1',\ldots, x_n'\}$ of $T^{*}$ given by
\[
x_1=e_{1}+ e_{1}',\; x_1'=e_1-e_{1}',\; x_i=e_{1}-e_{i},\; x_j'= e'_{1}-e_{j}',\; i,j,>1,
\]
The standard basis can be expressed in terms of this basis over $\Q$ as
\[
e_i=e_1-x_i,\; e_j'=e_1'-x_j',\text{ where }e_1=\tfrac{1}{2}(x_1+x_1') \text{ and }e_1'=\tfrac{1}{2}(x_1-x_1').
\]
The group $S^{2}(T_{sc}^*)^{W}$ is generated by $q=\sum_{i=1}^{m} e_{i}^{2}$ and $q'=\sum_{j=1}^{n} e_{j}'^{2}$. 
Therefore, for any $\phi\in S^{2}(T_{sc}^*)^{W}$ there exist $d, d'\in \Z$ such that $\phi=dq +d'q'=$
{\small \[
d\Big(\tfrac{1}{4}(x_1+x_1')^2+\sum_{i>1} (\tfrac{1}{2}(x_1+x_1')-x_i)^2\Big)+d'\Big(\tfrac{1}{4}(x_1-x_1')^2+\sum_{j>1} (\tfrac{1}{2}(x_1-x_1')-x_j')^2\Big).
\]}
The form $\phi$ has integer coefficients at $x_1x_1'$, $x_ix_j'$, $x_i^2$, $x_j'^2$, $i,j>1$ and it has coefficient $\tfrac{1}{4}(dm+d'n)$ at $x_1^2$ and at $x_1'^2$. Hence, 
\begin{equation}\label{QGC}
Q(G)=\{dq+d'q'\,|\,  dm+d'n\equiv 0 \bmod 4\}.
\end{equation}

Consider the subgroup $\Dec(G)$ of decomposable invariants of $G$.
As in the proof of \cite[Lemma 14.2]{GMS}, since the Weyl group of $G$ contains normal subgroups $(\Z/2\Z)^m$ and $(\Z/2\Z)^n$ generated by sign switching, we conclude that the coefficient at each $e_i$ in the expansion of $q_\chi$ is divisible by 2, hence, $\Dec(G)\subseteq 2\Z q\oplus 2\Z q'$. Since $c_{2}(\bar\rho(2e_{1}))=4q$ and $c_{2}(\bar\rho(2e_{1}'))=4q'$, we have $4\Z q\oplus 4\Z q'\subseteq \Dec(G)$.

Assume $n\equiv m\equiv 0 \bmod 2$. Since $c_{2}(\bar\rho(x_2))=2(m-1)q$ and $c_{2}(\bar\rho(x_2'))=2(n-1)q'$, 
we obtain $\Dec (G)=2\Z q\oplus 2\Z q'$ and $Q(G)=\{dq+d'q'\mid d\tfrac{m}{2}+d'\tfrac{n}{2}\equiv 0\bmod 2\}$. Hence,
{\small \[
\Inv(G) \simeq
\begin{cases}
(\Z/2\Z) q \oplus (\Z/2\Z) q' & \text{ if }n\equiv m\equiv 0\bmod 4, \\
(\Z/2\Z) q & \text{ if }m\equiv 0\not\equiv n \bmod 4, \\
(\Z/2\Z) q' & \text{ if }m\not\equiv 0\equiv n\bmod 4,\\
(\Z/2\Z) (q+q') & \text{ if }m\not \equiv 0 \not\equiv n \bmod 4.
\end{cases}
\]}

Assume both $n$ and $m$ are odd. If $n\equiv -m\bmod 4$, then $Q(G)=\{dq+d'q'\mid d\equiv d'\bmod 4\}\simeq \Z/4\Z (q+q')$. Since $c_{2}(\bar\rho(x_1))=2nq+2mq'$, $\Dec(G)=2\Z (q+q')\oplus 2\Z (q-q')$ and, therefore, $\Inv(G)\simeq (\Z/2\Z)(q+q')$. Similarly, if $n\equiv m\bmod 4$, then
$\Inv(G)\simeq (\Z/2\Z)(q-q')$.

Finally, assume $n$ is odd and $m$ is even. If $m\equiv 0\bmod 4$, then $Q(G)=\Z q\oplus 4\Z q'$. Since $c_{2}(\bar\rho(x_2))=2(m-1)q$, $\Dec(G)=2\Z q\oplus 4\Z q'$, hence, $\Inv(G)\simeq (\Z/2\Z)q$. If $m\not\equiv 0\bmod 4$, then $d'$ is even and 
\[
Q(G)=\{dq+d'q' \mid d+\tfrac{d'}{2}\equiv 0\bmod 2\}=\{(\bar 0,\bar 0),(\bar 2,\bar 0),(\bar 1,\bar 2),(-\bar 1,\bar 2)\},
\]
where $(\bar d,\bar d')$ denotes $dq+d'q'$ modulo 4. Since $\Dec(G)=2\Z q +4\Z q'$, we have $\Inv(G)\simeq (\Z/2\Z) (q+2q')$.

As for semi-decomposable invariants, consider an element (cf. with $h_{1,i}$ of Definition~\ref{dfn:genlist})
\begin{equation}\label{eq:typeCgen}
z=\tfrac{n}{gcd(m,n)}\bar\rho(\omega_1)-\tfrac{m}{gcd(m,n)}\bar\rho(\omega_1').
\end{equation}
By definition, we have $y=e^{e_{1}}z\in \Z[T^{*}]\cap I_{sc}^{W}$ and we obtain
\[
c_{2}(y)=c_2((1+(e^{e_1}-1))z)=c_{2}(z)=\tfrac{n}{gcd(m,n)}q-\tfrac{m}{gcd(m,n)}q'
\]
where the second equality holds since $(e^{e_1}-1)z\in I_{sc}^3$.
The element $c_2(y)\in \Sdec(G)$ coincides with the generator of $\Inv(G)=\Z/2\Z$ except if $n\equiv m \equiv 0\bmod 4$.
So $\Sinv(G)=\Inv(G)=\Z/2\Z$ except if  $n\equiv m \equiv 0\bmod 4$.

Assume that $m\equiv n\equiv 0 \bmod 4$. Then $\{q, q'\}$ are generators of the group of indecomposable invariants. Consider an arbitrary element $x\in \Z[T^{*}]\cap I_{sc}^{W}$ and the ring homomorphism 
\[\phi\colon \Z[T_{sc}^*]\to \Z[T_{sc}^*/T^{*}]=\Z[t]/(t^{2}-2t)\]
given by $\phi(1-e^{-w_{odd}})=\phi(1-e^{-w_{odd}'})=t$ and $\phi(1-e^{-w_{even}})=\phi(1-e^{-w_{even}'})=0$. Write 
{\small \[
x=\sum_{i=1}^m (d_i+\delta_i)\bar\rho(\omega_i)+\sum_{j=1}^n (d_i'+\delta_i')\bar\rho(\omega_i'),\quad d_i,d_j'\in \Z,\; \delta_i,\delta_j'\in I_{sc}.
\]}
Since $\ker \phi \supset \Z[T^{*}]\cap I_{sc}^{W}$, we obtain
{\small \[
0=\phi(x)=\sum_{odd\; i}2^i{m\choose i}(d_i+2s_i)t+\sum_{odd\; j}2^j{n\choose j}(d_j'+2s_j')t.
\]}
Observe that if $2^r\mid m$ and $i$ is odd, then $2^r\mid {m\choose i}$. Dividing by the 2-primary part $2^r$ of the greatest common divisor of all the coefficients we obtain
\[
(m/2^{r-1})d_1+(n/2^{r-1})d_1'\equiv 0\bmod 2,
\]
where $2^{r-1}=g.c.d.(v_2(n),v_2(m))$ is the g.c.d. of the 2-primary parts. Therefore,
{\small \begin{equation}\label{eq:divis}
\begin{cases} d_{1}+d_{1}'\equiv 0\bmod 2 & \text{ if } v_2(m)=v_2(n) \\
d_{1}\equiv 0\bmod 2 & \text{ if } v_2(m)<v_2(n).\end{cases}
\end{equation}}
We then have
{\small \[
c_2(x)=(\sum_{i=1}^m 2^{i-1}{m-1\choose i-1}d_i)q+(\sum_{j=1}^n 2^{j-1}{n-1 \choose j-1}d_j')q'.
\]}
So
$
c_2(x)\equiv d_1q+d_1'q'\bmod \Dec(G)$, where $d_1$ and $d_1'$ satisfy  \eqref{eq:divis}.

Since {\small \[
c_2(y)=\begin{cases} q+q' \bmod \Dec(G) & \text{ if }v_2(n)=v_2(m) \\ q' \bmod \Dec(G) & \text{ if }v_2(m)<v_2(n)\end{cases}\]}
we conclude that $c_2(y)$ is also a generator of $\Sinv(G)\simeq \Z/2\Z$.
\end{proof}

We now present a generalization of the previous proposition, which in turn determine both indecomposable and semi-decomposable subgroups for an arbitrary semisimple group of type ${\rm C}$.

\begin{corollary}\label{cor:typec}
	(1) Let $G=\prod_{i=1}^{m}\gPGSP_{2n_{i}}$, $m,n_{i}\ge 1$. Then, {\small \[\Inv(G)\simeq (\Z/2\Z)^{\oplus k},\]} where $k$ is the number of $n_{i}$'s which are divisible by~$4$, and $\Sdec(G)=\Dec(G)$, i.e., each semi-decomposable invariant is decomposable.	
	
		(2) Let $G=(\prod_{i=1}^{m}\gSP_{2n_{i}})/\gmu_2$, $m, n_{i}\geq 1$, where $\gmu_{2}$ is the diagonal subgroup. Then,  
		{\small \[\Inv(G)=\begin{cases}(\Z/2\Z)^{\oplus m} & \text{ if } \forall n_{i}\equiv 0\bmod 4,\\
			(\Z/2\Z)^{\oplus m-1} & \text{ otherwise,} \end{cases}\]}
		and 
		{\small \[\Sinv(G)=\begin{cases}(\Z/2\Z)^{\oplus m-1} & \text{ if } \forall n_{i}\equiv 0  \text{ or } \forall n_{i}\not\equiv 0 \bmod 4,\\
			(\Z/2\Z)^{\oplus m-2} & \text{ otherwise.} \end{cases}\]}
			
		(3) Let $G=(\prod_{i=1}^{m}\gPGSP_{2n_i})\times (\prod_{i=1}^{m'}\gSP_{2n_{i}'})$, $n_{i}, n_{i}'\geq 1$. Then, {\small \[\Inv(G)=(\Z/2\Z)^{\oplus k} \text{ and } \Sdec(G)=\Dec(G),  \]}
		where $k$ is the number of $n_{i}$'s which are divisible by~$4$.
\end{corollary}

\begin{proof}
(1) Let $G_1=\gPGSP_{2m}$ and $G_2=\gPGSP_{2n}$, $m, n\geq 1$. It suffices to consider the case $G=G_{1}\times G_{2}$ since the same arguments can be easily adapted to prove the case of three and more groups. We simply set $d'=0$ (resp. $d=0$) in \ref{QGC}. Then, we have $Q(G_{1})=4/\gcd(4,m)\Z q$ (resp. $Q(G_{2})=4/\gcd(4,n)\Z q'$). Similarly, by the proof of Proposition \ref{prop:typec} we get $\Dec(G_{1})=4/\gcd(2,m)\Z q$ (resp. $\Dec(G_{2})=4/\gcd(2,n)\Z q'$) (\cite[\S4b]{Mer}). By~\eqref{eq:product} the answer for $\Inv(G)$ then follows.

As for semi-decomposable invariants, by~\eqref{eq:inclus} and
$\Dec(G)=Q(G)$ for $n\not\equiv 0\bmod 4$,
it suffices to consider the case $n\equiv m\equiv 0 \bmod 4$. We follow arguments used in \cite[\S 3c]{MNZ}.

Let $x\in \Z[T^*]\cap I_{sc}^W$ be an arbitrary element. Write 
\begin{equation*}
x=\sum_{i=1}^{m} (d_{i}+\delta_{i})\bar{\rho}(\omega_{i})+\sum_{j=1}^{n} (d_{j}'+\delta_{j}')\bar{\rho}(\omega_{j}')
\end{equation*}
for some $d_{i}, d_{j}'\in \Z$ and $\delta_{i}, \delta_{j}'\in I_{sc}$. 
Consider the ring homomorphism induced by the quotient map $T^*_{sc} \to T^*_{sc}/T^*$ 
\[
\phi\colon \Z[T_{sc}^*]\longrightarrow \Z[T_{sc}^*/T^*]=\Z[\Lambda_{w}/\Lambda_{r}]\otimes \Z[\Lambda_{w}'/\Lambda_{r}']=\Z[t,t']/(t^{2}-2t, t'^{2}-2t').
\] It is given by \[
\phi(1-e^{-\omega_{odd}})=t,\; \phi(1-e^{-\omega_{odd}'})=t'\text{ and }\phi(1-e^{-\omega_{even}})=\phi(1-e^{-\omega_{even}'})=0.\] 

Since $x\in \Z[T^*]$, $\phi(x)=0$. Moreover, we have $\phi(I_{sc})=(t,t')$ and $\phi(\bar{\rho}(\omega_{i}))={|W(\omega_i)|}\cdot  t$, $\phi(\bar{\rho}(\omega_{i}'))=|W(\omega_i')|\cdot t'$. Combining these facts we obtain
\begin{align*}
	0 &=\sum_{i=1}^{m} (d_{i}+\phi(\delta_{i}))\phi(\bar{\rho}(\omega_{i}))+\sum_{j=1}^{n} (d_{j}'+\phi(\delta_{j}'))\phi(\bar{\rho}(\omega_{j}')) \\
	&=\sum_{odd\; i} |W(\omega_i)|\cdot (d_{i}+s_it+s_i't') t+\sum_{odd\; j} |W(\omega_j')|\cdot (d_{j}'+r_jt+r_j't') t',
\end{align*}
for some $s_i, s_i',r_j,r_j'\in \Z$.
Since $|W(\omega_i)|=2^i {m \choose i}$, collecting the coefficients at $t$ and $t'$, we get
\[
d_1+2s_1\equiv d_1'+2r_1'\equiv 0 \bmod 2.
\]
Hence, both $d_1$ and $d_1'$ are even.

We now compute $c_2(x)$. Since $c_2(I_{sc}^3)=0$ and $c_2(\bar\rho(\omega_i))=2^{i-1}{m-1 \choose i-1}q$, we obtain
\[
c_2(x)=\sum_{i=1}^{m} d_{i}c_2(\bar{\rho}(\omega_{i}))+\sum_{j=1}^{n} d_{j}'c_2(\bar{\rho}(\omega_{j}'))=2sq+2rq',\;\text{ for some } r,s\in \Z.
\]
Therefore, $\Sdec(G_1)\subseteq 2\Z q=\Dec(G_1)$ and $\Sdec(G_2)\subseteq 2\Z q'=\Dec(G_2)$.

(2) This immediately follows from the same argument as in Proposition \ref{prop:typec}.

(3) As $\Dec(\gSP_{2n})=\Sdec(\gSP_{2n})=Q(\gSP_{2n})$ for any $n\geq 1$ (\cite[Theorem 14.3]{GMS}), the same argument as in $(1)$ shows that $\Sdec(G)=\Dec(G)$ and the result for the indecomposable subgroup.\end{proof}

\section{Type {\rm D}}

In this section we calculate the groups of indecomposable and semi-decomposable invariants for an arbitrary product of simply-connected simple groups of type ${\rm D}$ modulo the (diagonal) central subgroups. We first consider the groups of index $2$ and $4$.
\begin{proposition}\label{Ddiagonal}
Let $G=(\gSpin_{2m}\times \gSpin_{2n})/\gmu$, where $m, n\geq 4$ and $m+n$ is even, and  $\gmu$ is a diagonal subgroup of $G$. Then 
		{\small \[\Inv(G)=\begin{cases}\Z/4\Z & \text{ if } \gmu\simeq \gmu_{4},\\
		\Z/2\Z & \text{ if } \gmu\simeq \gmu_{2},\end{cases} \quad\text{ and }\quad \Sinv(G)=\begin{cases}\Z/2\Z & \text{ if } \gmu\simeq \gmu_{4},\\
		0 & \text{ if } \gmu\simeq \gmu_{2}.\end{cases}\]}
		
\end{proposition}
\begin{proof}
Observe that there is a unique diagonal subgroup $\gmu\simeq \gmu_{4}$ in the case where $m$ and $n$ are odd and there are two  different diagonal subgroups $\gmu\simeq \gmu_{2}$ : $\gmu\subseteq \gmu_{4}^{2}$ if both $m$ and $n$ are odd and $\gmu\subseteq \gmu_{2}^{4}$ otherwise. First, assume that $\gmu\simeq \gmu_{4}$. Then, by \ref{par:char} the character group of the split maximal torus $T$ of $G$ is given by
{\small \begin{equation}\label{chT1}
T^{*}=\big\{\sum_{i=1}^{m}a_{i}\omega_{i}+\sum_{j=1}^{n}a_{i}'\omega_{i}' \,\big|\,  a_{m-1}+3a_{m}+2\sum_{i=1}^{\tfrac{m-1}{2}}a_{2i-1}\equiv 3a_{n-1}'+a_{n}'+2\sum_{j=1}^{\tfrac{n-1}{2}}a_{2j-1}' \bmod 4 \big\}.
\end{equation}}
Write $\sum_{i=1}^{m}a_{i}\omega_{i}+\sum_{j=1}^{n}a_{i}'\omega_{i}'=\sum_{i=1}^{m}b_{i}e_{i}+\sum_{j=1}^{n}b_{i}'e_{i}'$ in terms of the standard basis vectors $\{e_{1},\cdots, e_{m}, e'_{1},\cdots, e_{n}'\}$ of $T_{sc}^*=\Z^{m}\oplus \Z^{n}$. Then, the relation in (\ref{chT1}) is equivalent to
{\small \begin{equation}\label{chT2}
2(b_{m-2}+b_{m-1}+b_{m}+\sum_{i=1}^{\tfrac{m-3}{2}}b_{2i-1}-b_{2i})\equiv 2(b_{n-2}'+b_{n-1}'-b_{n}'+\sum_{j=1}^{\tfrac{n-3}{2}}b_{2j-1}-b_{2j}) \bmod 4.
\end{equation}}
Using \eqref{chT1} we choose the following basis $\{x_{1},\ldots, x_{m},x_{1}',\ldots, x'_{n}\}$ of $T^{*}$:
{\small \begin{equation}\label{xid}\begin{split}&x_{2i-1}=\omega_{2i-1}+2\omega_{n}',\,\, x_{2k}=\omega_{2k},\,\, x_{m-1}=\omega_{m-1}+\omega_{n}',\,\, x_{m}=\omega_{m}+3\omega_{n}',\\
&x_{2j-1}'=\omega_{2j-1}'+2\omega_{n}',\,\, x_{2l}'=\omega_{2l}',\,\, x_{n-1}'=\omega_{n-1}'+\omega_{n}',\,\, x_{n}'=4\omega_{n}'
\end{split}\end{equation}}
for $1\leq i\leq (m-1)/2$, $1\leq k\leq (m-3)/2$, $1\leq j\leq (n-1)/2$, and $1\leq l\leq (n-3)/2$.

Let $\psi$ be a quadratic form on (\ref{xid}) with integer coefficients. As the group $S^{2}(T^*_{sc})^{W}$ is generated by the normalized Killing forms {\small \[\begin{split}q:=\sum_{i=1}^{m}\omega_{i}^{2}-2\sum_{i=1}^{m-2}\omega_{i}\omega_{i+1}-2\omega_{m-2}\omega_{m}=(\sum_{i=1}^{m}e_{i}^{2})/2,
\\
q':=\sum_{j=1}^{n}\omega_{j}'^{2}-2\sum_{j=1}^{n-2}\omega_{j}'\omega_{j+1}'-2\omega_{n-2}'\omega_{n}'=(\sum_{j=1}^{n}e_{j}'^{2})/2,&
\end{split}
\]}
from the equation $\phi=dq+d'q'$ we get 
\[\phi=(\tfrac{md+nd'}{8})x_{n}'^{2}+\psi,\]
where $\psi$ is a quadratic form on (\ref{xid}) with integer coefficients. Therefore,   
\begin{equation}\label{stgenDodd}
Q(G)=\{dq+d'q'\,|\,  md+nd'\equiv 0 \bmod 8\}.
\end{equation}

We show that $\Dec(G)=4\Z(q-q')+4\Z(q+q')$. First, by (\ref{chT1}) we see that all elements $c_{2}(\rho(2\omega_1))=-8q$, $c_{2}(\rho(2\omega_1'))=-8q'$, $c_{2}(\rho(\omega_1+\omega_{1}'))=-4nq-4mq'$ are contained in $\Dec(G)$, thus $4(q-q'), 4(q+q')\in \Dec(G)$. On the other hand, since $\Dec(G)$ is generated by $c_{2}(\rho(\lambda))$ for all $\lambda\in T^{*}$ and $W$ contains normal subgroups $(\Z/2\Z)^{m-1}$ and $(\Z/2\Z)^{n-1}$, we see that the coefficient at each $e_{i}$ in the expansion of $c_{2}(\rho(\lambda))$ is divisible by 2, thus, $\Dec(G)\subseteq 4\Z q\oplus 4\Z q'$ (note that by (\ref{stgenDodd}) $4q, 4q'\not\in \Dec(G)$). Hence, {\small \begin{equation*}
 \Inv(G)=
\Z/4\Z(\tfrac{n}{\gcd(m,n)}q-\tfrac{m}{\gcd(m,n)}q').
\end{equation*}}

Now we show that $\Sdec(G)=2\Z(\tfrac{n}{\gcd(m,n)}q-\tfrac{m}{\gcd(m,n)}q')$. First of all, as 
\begin{equation}\label{Dfundweight}
c_{2}(\bar\rho(\omega_{i}))=2m_{i}q \text{ and }  c_{2}(\bar\rho(\omega_{j}'))=2m_{j}'q'
\end{equation}
for all $1\leq i\leq m$ and $1\leq j\leq n$ and for some $m_{i}, m_{j}'\in \Z$, by the same argument as in the proof of Proposition \ref{propB} we obtain $\tfrac{n}{\gcd(m,n)}q-\tfrac{m}{\gcd(m,n)}q'\not\in \Sdec(G)$. On the other hand, consider an element 
\begin{equation*}
	z=\tfrac{n}{gcd(m,n)}\bar\rho(\omega_1)-\tfrac{m}{gcd(m,n)}\bar\rho(\omega_1').
\end{equation*}
Then, by (\ref{chT2}) we obtain $y:=e^{e_{1}}z\in \Z[T^{*}]\cap I_{sc}^{W}$. Therefore, as in the proof of Proposition \ref{prop:typec}, we have $c_2(y)=2(\tfrac{n}{\gcd(m,n)}q-\tfrac{m}{\gcd(m,n)}q')\in \Sdec(G)$. Hence, the result for $\Sinv(G)$ follows.

Now we assume that $\gmu\simeq \gmu_{2}$, so that $\gmu\subseteq \gmu_{4}^{2}$ if $m$ and $n$ are odd and $\gmu\subseteq \gmu_{2}^{4}$ otherwise. In both cases, the corresponding character group of $T$ is given by
{\small \[T^{*}=\big\{\sum_{i=1}^{m}a_{i}\omega_{i}+\sum_{j=1}^{n}a_{i}'\omega_{i}' \,\big|\,  a_{m-1}+a_{m}\equiv a_{n-1}'+a_{n}' \bmod 2 \big\}.\]}
By applying the same argument in the above $\gmu\simeq \gmu_{4}$ case, we obtain 
\begin{equation*}
Q(G)=\{dq+d'q'\,|\,  d+d'\equiv 0 \bmod 2\}.
\end{equation*}
Since $c_{2}(\rho(\omega_1))=-2q$ and $c_{2}(\rho(\omega_1'))=-2q'$, we have $2q, 2q'\in \Dec(G)$. Moreover, by (\ref{Dfundweight}), we get $q-q'\not\in \Sdec(G)$, thus $\Sdec(G)=\Dec(G)=2\Z q\oplus 2\Z q'$ and $\Inv(G)=\Z/2\Z(q-q')$.\end{proof}

We obtain the following generalization of the previous proposition. Together with Remark \ref{rem:typeD}, they determine the groups of indecomposable and semi-decomposable invariants for an arbitrary product of simply-connected simple groups of type ${\rm D}$ modulo the central subgroups $\gmu_{2}$.

\begin{corollary}\label{cor:typeD}
Let $G=(\prod_{i=1}^{m}\gSpin_{2n_{i}})/\gmu$, $n_{i}\geq 4$, $m\geq 1$, where either all $n_{i}$ are even or odd, and $\gmu$ is a diagonal subgroup. Then, 
	{\small\[\Inv(G)=\begin{cases}(\Z/4\Z)^{m-1} & \text{ if } \gmu\simeq \gmu_{4},\\
	(\Z/2\Z)^{m-1} & \text{ if } \gmu\simeq \gmu_{2},\end{cases} \text{ and  } \,\Sinv(G)=\begin{cases}(\Z/2\Z)^{m-1} & \text{ if } \gmu\simeq \gmu_{4},\\
	0 & \text{ if } \gmu\simeq \gmu_{2}.\end{cases}\]}
\end{corollary}

\begin{proof}
By the same argument as in Proposition~\ref{Ddiagonal}, we obtain
\[Q(G)=\begin{cases}\{\sum_{i=1}^{m}d_{i}q_{i}\,|\, \sum n_{i}d_{i}\equiv 0 \bmod 8\} & \text{ if } \gmu\simeq \gmu_{4},\\  
\{\sum_{i=1}^{m}d_{i}q_{i}\,|\, \sum d_{i}\equiv 0 \bmod 2\} & \text{ if } \gmu\simeq \gmu_{2},\end{cases}
\]
where $q_{1},\ldots, q_{m}$ are the corresponding normalized Killing forms of $\gSpin_{2n_{i}}$. Similarly, 
\[\Dec(G)=\begin{cases} \bigoplus_{i=2}^{m}4\Z (q_{1}-q_{i})\oplus 4\Z (q_{1}+q_{2}) & \text{ if } \gmu\simeq \gmu_{4},\\  
\bigoplus_{i=1}^{m} 2\Z q_{i}& \text{ if } \gmu\simeq \gmu_{2},\end{cases}
\]
thus, the factor groups follow. Following Proposition \ref{Ddiagonal}, we see that 
\[\Sdec(G)/\Dec(G)=\bigoplus_{i=2}^{m}\Z/2\Z(\tfrac{n_{i}}{\gcd(n_{1},n_{i})}q_{1}-\tfrac{n_{1}}{\gcd(n_{1},n_{i})}q_{i})\]
if $\gmu\simeq \gmu_{4}$ and $q_{1}-q_{i}\not\in \Sdec(G)$ for all $2\leq i\leq m$ if $\gmu\simeq \gmu_{2}$, which completes the proof.
\end{proof}
\begin{remark}\label{rem:typeD}
(1) Let $G'=G\times (\prod_{i=1}^{m'}\gSpin_{2n_{i}'})$, where $G$ is the group from Corollary~\ref{cor:typeD}. Then, similar to the proof of Proposition \ref{Ddiagonal} one can show that $\Sinv(G')=\Sinv(G)$. Moreover, by \cite[Theorem 15.4]{GMS} $\Inv(G')=\Inv(G)\oplus (\Z/2\Z)^{\oplus m'}$.
	
(2) Let $G'=G\times (\prod_{i=1}^{m'}\gSpin_{2n_{i}'})$, where $G$ is either $\gSO_{2m}$ or $\gHSpin_{2m}$. If $G=\gSO_{2m}$, then by \cite[\S 15]{GMS} $\Dec(G)=\Dec(\gSpin_{2m})$, thus $\Dec(G')=\Sdec(G')$. Similarly, if $G=\gHSpin_{2m}$, then it follows from \cite[\S 5]{BR} and \cite[\S 3d]{MNZ} that $\Dec(G')=\Sdec(G')$. One can also easily compute the indecomposable groups.
\end{remark}

\section{The $\gPGO_8$-case}

In the present section we use the techniques developed in section~\ref{sec:gen} to give a direct proof of the main result of \cite[Appendix]{MNZ}.

In this section, $G=\gPGO_8$ that is an adjoint group of Dynkin type ${\rm D}_4$.
The weight lattice of type ${\rm D}_4$ can be constructed as follows. We first take a $\Q$-vector space with basis $e_1,\ldots, e_4$.
Then $\Lambda$ has the following $\Z$-basis consisting of fundamental weights:
{\small \[
\omega_1=e_1,\;  \omega_2=e_1+e_2,\; \omega_3=(e_1+e_2+e_3-e_4)/2,\; \omega_4=(e_1+e_2+e_3+e_4)/2.
\]}
So the coordinates of elements of $\Lambda$ are either all integers or half-integers.

The group $T^*$ consists of all points such that all coordinates are integers, and the sum of coordinates is divisible by $2$. 
We have $\Lambda/T^*=\Z/2\Z\oplus \Z/2\Z$ with elements $\bar 0=(0,0)$, $\bar \omega_1=(0,1)$, $\bar\omega_3=(1,0)$, $\bar\omega_4=(1,1)$.
The quotient map $\Lambda \to \Lambda/T^*$ induces a grading on $\Lambda$ and, hence, on $\Z[\Lambda]$.
We denote by $\Lambda^{(0,0)},\Lambda^{(0,1)},\Lambda^{(1,0)},\Lambda^{(1,1)}$ the respective homogeneous components.
Each polynomial $f\in \Z[\Lambda]$ can be split into a sum of its homogeneous components, which we will denote by $f^{(i,j)}$.
Denote the orbits in $\Z[\Lambda]$ by $\rho(\omega_1),\ldots,\rho(\omega_4)$ and the augmented orbits by $\rho_1,\ldots,\rho_4$ as in section~\ref{sec:gen}.

\begin{lemma}\label{d4lemma}
Let $f_1,\ldots,f_4\in \Z[\Lambda]$ be such that $f_1\rho_1+\ldots+f_4\rho_4\in \Z[T^*]$.
Then, for each element $i\in \Lambda/T^*$ except for $i=(0,1)$ the sum of coefficients of $f_1^{(i)}$ is even.
\end{lemma}

\begin{proof}
Consider a subgroup 
\[
\Lambda'=\{x_1e_1+x_2e_2+x_3e_3+x_4e_4\in T^*\mid x_2+x_3+x_4\text{ and }x_1 \text{ is even}\}.
\]
We have $\Lambda/\Lambda'=\Z/2\Z\oplus \Z/4\Z$ with a generator $\bar \omega_1$ of order $2$ and $\bar \omega_4$ of order $4$.

Set $R=\Z/4\Z$. Consider a natural map $\Z[\Lambda]\to R[\Lambda/\Lambda']$ given by $f\mapsto \bar{f}$.
Since $\Lambda'\subset T^*$, $R[\Lambda/\Lambda']$ is also a $\Lambda/T^*$-graded algebra, and this map preserves the grading.
By definition, the sum of coefficients of $f$ modulo $4$ is 
the sum of coefficients of $\bar{f}$.
So, it is sufficient to prove that 
for each element $i\in \Lambda/T^*$ except for $i=(0,1)$, the sum of coefficients of 
$\bar{f}_1^{(i)}=\bar{f}_1^{(i)}$ is even.

Since $\bar\rho_1=2\overline{e^{e_1}}+2\overline{e^{e_2}}$, $\bar \rho_2=\bar\rho_3=\bar \rho_4=0$ in $R[\Lambda/\Lambda']$ and $\bar{f}_1\bar{\rho}_1+\ldots+\bar{f}_4\bar{\rho}_4 \in R[T^*/\Lambda']$, we obtain that $\bar{f}_1\bar{\rho}_1\in R[T^*/\Lambda']$.
Since $\bar\rho_1 \in \Z[\Lambda^{(0,1)}]$, 
for each $i\in \Lambda/T^*$, $i\ne 0$, we have 
\[
0=(\bar{f}_1\bar{\rho}_1)^{(i)}=\bar{f}_1^{(i-(0,1))}\bar\rho_1=
2\overline{e^{e_1}}(1+\overline{e^{e_1+e_2}})\bar{f}_1^{(i-(0,1))}.\]
Therefore, for each $j\in \Lambda/T^*$, $j\ne (0,1)$, all coefficients 
of $(1+\overline{e^{e_1+e_2}})\bar{f}_1^{(j)}$ are divisible by $2$.

Observe that the classes of $0$ and of $\omega_2$ in $\Lambda/\Lambda'$
are all elements of $\Lambda/\Lambda'$ that belong to $T^*/\Lambda'$.
So, if $\lambda\in \Lambda/\Lambda'$, and $f\in R[\Lambda/\Lambda']$, 
then the coefficient of $(1+\overline{e^{e_1+e_2}})f$ at $e^{\lambda}$ is
the sum of coefficients of $f$ in front of $e^{\lambda'}$ for all $\lambda'\in \Lambda/\Lambda'$ 
such that $\lambda= \lambda' \bmod T^*$ (there are exactly two such $\lambda'$, one of them is 
$\lambda$, the other is $\lambda+e_1+e_2$). If $f$ is a homogeneous polynomial of 
degree $j\in \Lambda/T^*$, and $\lambda\in \Lambda/\Lambda'$ is mapped to $j$ by the natural projection 
$\Lambda/\Lambda'\to \Lambda/T^*$, then the coefficient of 
$(1+\overline{e^{e_1+e_2}})f$ in front of $e^{\lambda}$ is the sum of all coefficients of $f$.
Therefore, the sum of all coefficients of $\bar{f}_1^{(j)}$ 
is divisible by 2 for each $j\in \Lambda/T^*$, $j\ne (0,1)$.
\end{proof}

\begin{lemma}\label{lem:othereven}
Let $f_1,\ldots,f_4\in \Z[\Lambda]$ be such that $f_1\rho_1+\ldots+f_4\rho_4\in \Z[T^*]$.
Then, 

(1) for each element $i\in \Lambda/T^*$ except for $i=(1,0)$ the sum of coefficients of $f_3^{(i)}$ is even;

(2) for each element $i\in \Lambda/T^*$ except for $i=(1,1)$ the sum of coefficients of $f_4^{(i)}$ is even.
\end{lemma}

\begin{proof}
(1) Consider an automorphism $\psi$ of $\Lambda$ (induced by an outer automorphism of $\gPGO_8$) that interchanges 
$\omega_1$ and $\omega_3$ and keeps $\omega_2$ and $\omega_4$ invariant. It preserves $T^*$. So, 
it also acts on $\Lambda/T^*=\Z/2\oplus \Z/2$ by interchanging $(0,1)$ and $(1,0)$ (and keeping $(0,0)$ and $(1,1)$).

By definition $\psi$ maps the graded components of a polynomial $f\in \Z[T^*]$ to 
graded components of $\psi(f)$, more precisely, $\psi(f^{(i)})=\psi(f)^{\psi(i)}$. In particular, 
$\psi(f)^{(0,1)}=\psi(f^{(1,0)})$.

Since $\psi$ interchanges $\rho_1$ and $\rho_3$, and keeps $\rho_2$ and $\rho_4$ unchanged,
\[
\psi(f_1\rho_1+\ldots+f_4\rho_4)=\psi(f_1)\rho_3+\psi(f_2)\rho_2+\psi(f_3)\rho_1+\psi(f_4)\rho_4.\]
Finally, observe that the sum of coefficients of $\psi(f)$ is the same as the sum of coefficients of $f$. 
We then apply Lemma~\ref{d4lemma}.

(2) The proof is completely similar to the proof of the previous case, the only difference is that now $\psi$ should interchange $\omega_1$
and $\omega_4$ and keep $\omega_2$ and $\omega_3$ unchanged.
\end{proof}

\begin{proposition}\label{prop:evensum}
Let $f_1,\ldots,f_4\in \Z[\Lambda]$ be such that $f_1\rho_1+\ldots+f_4\rho_4\in \Z[T^*]$.
Then the sum of all coefficients of $f_i$ is even for $i=1,3,4$.
\end{proposition}

\begin{proof}
Set
$\omega_{(0,1)}=\omega_1$,
$\omega_{(1,0)}=\omega_3$, 
$\omega_{(1,1)}=\omega_4$,
and
$\omega_{(0,0)}=\omega_2$.
Then the class of $\omega_{(i,j)}$ in $\Lambda/T^*$ is precisely $(i,j)$.

Set $R=\Z/16\Z$ and consider the natural map $\Z[\Lambda]\to R[\Lambda/T^*]$ given by $f\mapsto \bar{f}$.
Since $f_1\rho_1+\ldots+f_4\rho_4\in \Z[T^*]$, 
$\bar{\rho}_i=8\overline{e^{\omega_i}}-8$, $i=1,3,4$ and $\bar{\rho}_2=24\overline{e^{\omega_2}}-24=0$,
\[\bar f_1\bar \rho_1+\bar f_3\bar \rho_3+\bar f_4\bar \rho_4\text{ is a constant in }R[\Lambda/T^*].\]

For $i=1,3,4$ and $(j,k)\in \Lambda/T^*$, denote by $c_i^{(j,k)}$ the sum of coefficients of $f_i^{(j,k)}$ modulo $16$.
Then $\bar{f}_i^{(j,k)}=c_i^{(j,k)}\overline{e^{\omega_{(j,k)}}}$.
By Lemma \ref{d4lemma} and~\ref{lem:othereven}, all numbers $c_i^{(j,k)}$ are even, except for, possibly, 
$c_1^{(0,1)}$,
$c_3^{(1,0)}$, 
and 
$c_4^{(1,1)}$.
Observe that if $c_i^{(j,k)}$ is even, then 
$c_i^{(j,k)}\overline{e^{\omega_{(j,k)}}}\bar{\rho}_i=0$ in $R[\Lambda/T^*]$ since $\bar{\rho}_i$
is divisible by $8$.
Therefore, both expressions
\[c_1^{(0,1)}\overline{e^{\omega_1}}\bar{\rho}_1+
c_3^{(1,0)}\overline{e^{\omega_3}}\bar{\rho}_3+
c_4^{(1,1)}\overline{e^{\omega_4}}\bar{\rho}_4\quad \text{ and}\]
\[8c_1^{(0,1)}\overline{e^{\omega_1}}+
8c_3^{(1,0)}\overline{e^{\omega_3}}+
8c_4^{(1,1)}\overline{e^{\omega_4}}+
8(c_1^{(0,1)}+c_3^{(1,0)}+c_4^{(1,1)})\]
are constants in $R[\Lambda/T^*]$.

So, the coefficients 
$c_1^{(0,1)}$,
$c_3^{(1,0)}$, 
and 
$c_4^{(1,1)}$
are even which means 
that 
for all $i=1,3,4$ and $(j,k)\in \Lambda/T^*$, 
$c_i^{(j,k)}$ is even. But then the sum of all coefficients of $f_i$ is even for $i=1,3,4$.
\end{proof}

We now give a direct proof of the result obtained in \cite[Appendix]{MNZ} using a computer algorithm
\begin{corollary}\label{pgo8}
If $G=\gPGO_8$, then any semi-decomposable invariant of $G$ is decomposable.
\end{corollary}

\begin{proof}
Let $x\in \Z[T^*]\cap I^W_{sc}$. Similar to \cite[\S3c]{MNZ} we write
\[
x=\sum_{i=1}^4 (d_i+\delta_i)\bar\rho(\omega_i)\text{ for some }d_i\in \Z,\; \delta_i\in I_{sc}.
\]
Then $c_2(x)=\sum_{i=1}^4d_i c_2(\bar \rho_i)$.
By Proposition~\ref{prop:evensum} we have $d_1\equiv d_3\equiv d_4 \equiv 0\bmod 2$.
Since $c_2(\bar\rho_i)=2q$ for $i=1,3,4$ and $c_2(\bar\rho_2)=12q$ by \cite[\S3d]{MNZ},
we obtain that $c_2(x)\in 4\Z q=\Dec(G)$.
\end{proof}

\section{Type ${\rm E}$}

We now treat the exceptional cases. In the following we show that any semi-decomposable invariant is decomposable for semisimple groups of type $E_{6}$ and $E_{7}$.

\begin{lemma}\label{lem:Esemidec}
Let $G$ be a split semisimple group of type $E_{6}$ or $E_{7}$. Then, $\Sdec(G)=\Dec(G)$, i.e., each semi-decomposable invariant is decomposable.
\end{lemma}
\begin{proof}
We denote by $E_{6}^{ad}$ $($resp. $E_{7}^{ad}$$)$ a split simple adjoint group of type $E_{6}$ $($resp. $E_{7}$$)$ and by $E_{6}^{sc}$ $($resp. $E_{7}^{sc}$$)$ a split simple simply connected group of type $E_{6}$ $($resp. $E_{7}$$)$. We first consider the case where $G$ is a semisimple group of type $E_{6}$, i.e., $G=(E_{6}^{sc}\times \cdots \times E_{6}^{sc})/\gmu$ ($n$ copies of $E_{6}^{sc}$) for some central subgroup $\gmu$. Let $q_{1},\cdots, q_{n}$ be the corresponding normalized Killing forms for each copy of $E_{6}^{sc}$ in $G$. Since $\Dec(E^{sc}_{6})=\Dec(E_{6}^{ad})=6q_{i}$ by \cite[\S 4b]{Mer}, we have $\Dec(E_{6}^{sc}\times \cdots \times E_{6}^{sc})=\Dec(E_{6}^{ad}\times \cdots \times E_{6}^{ad})=6\Z q_{1}\oplus \cdots \oplus 6\Z q_{n}$. As $\Dec(E_{6}^{ad}\times E_{6}^{ad})\subseteq \Dec(G)\subseteq \Dec(E^{sc}_{6}\times E^{sc}_{6})$, we conclude that $\Dec(G)=6\Z q_{1}\oplus \cdots \oplus 6\Z q_{n}$.

Now we show $\Sdec(G)\subseteq \Dec(G)$. Similar to the proof of Proposition \ref{propB}, we consider an arbitrary element $x\in I_{sc}^W$:
\begin{equation*}
	x=\sum_{j=1}^{n}\sum_{i=1}^{6} (d_{ij}+\delta_{ij})\bar{\rho}(\omega_{ij})
\end{equation*}
for some $d_{ij}\in \Z$ and $\delta_{ij}\in I_{sc}$, where $\{\omega_{1j},\ldots, \omega_{6j}\}$ is the fundamental weights of each copy of $E_{6}^{sc}$. Since $c_2(I_{sc}^3)=0$, we obtain
\begin{equation}\label{E6xc2}
c_2(x)=\sum_{j=1}^{n}\sum_{i=1}^{6} d_{ij}c_2(\bar{\rho}(\omega_{ij})).
\end{equation}
By \cite[\S 2]{LaS}, each element $c_2(\bar{\rho}(\omega_{ij}))$ in (\ref{E6xc2}) iz contained in $6\Z q_{j}$. Hence, $\Sdec(G)\subseteq \Dec(G)$, so the equality holds. 

Let $G=(E_{7}^{sc}\times \cdots \times E_{7}^{sc})/\gmu$ ($n$ copies of $E_{7}^{sc}$) for some central subgroup $\gmu$. Then, the same argument together with $\Dec(E_{6}^{sc})=\Dec(E_{6}^{ad})=12q_{i}$ (\cite[\S 4b]{Mer}) shows that $\Dec(G)=\Sdec(G)=12\Z q_{1}\oplus \cdots \oplus 12\Z q_{n}$.\end{proof}

We determine the indecomposable groups for the groups of index $2$ (resp. $3$) of type $E_{6}$ (resp. $E_{7}$).
\begin{proposition}\label{prop:typeE}
	(1) Let $G= (E^{sc}_{6}\times \cdots \times E^{sc}_{6})/\gmu_{3}$ with $n\, (\geq 2)$ copies of a split simple simply connected group $E_{6}^{sc}$ of type $E_{6}$ and the diagonal subgroup $\gmu_{3}$. Then \[\Inv(G)=\Z/2\Z\oplus (\Z/6\Z)^{\oplus n-1}\text{ and }\Sdec(G)=\Dec(G).\]
		
		(2) Let $G= (E^{sc}_{7}\times \cdots \times E^{sc}_{7})/\gmu_{2}$ with $n\, (\geq 2)$ copies of a split simple simply connected group $E_{7}^{sc}$ of type $E_{7}$ and the diagonal subgroup $\gmu_{2}$. Then \[\Inv(G)=\Z/3\Z\oplus (\Z/12\Z)^{\oplus n-1}\text{ and }\Sdec(G)=\Dec(G).\]
\end{proposition}
\begin{proof}
By Lemma \ref{lem:Esemidec}, it suffices to compute the indecomposable groups.

(1) Assume that $n=2$. Then, by \ref{par:char} the character group of the split maximal torus $T$ of $G$ is given by
{\small \begin{equation}\label{E6char}T^{*}=\{\sum_{i=1}^{6}a_{i}\omega_{i}+\sum_{i=1}^{6}a_{i}'\omega_{i}' \,|\, a_{1}+a_{1}'+a_{5}+a_{5}'\equiv a_{3}+a_{3}'+a_{6}+a_{6}' \bmod 3 \}.
\end{equation}}
Choose a basis $\{x_{1},\ldots x_{6},x_{1}',\ldots, x_{6}'\}$ of (\ref{E6char}) as follows
{\small \begin{equation*}
	\begin{split}&
		x_{1}=\omega_{1}+\omega_{6}', x_{2}=\omega_{2}, x_{3}=\omega_{3}+2\omega_{6}', x_{4}=\omega_{4}, x_{5}=\omega_{5}+\omega_{6}', x_{6}=\omega_{6}+2\omega_{6}',\\
		&x_{1}'=\omega_{1}'+\omega_{6}', x_{2}'=\omega_{2}', x_{3}'=\omega_{3}'+2\omega_{6}', x_{4}'=\omega_{4}', x_{5}'=\omega_{5}'+\omega_{6}', x_{6}'=3\omega_{6}'.
	\end{split}
\end{equation*}}
Let $\phi$ be a quadratic form on $x_{i}$, $x_{i}'$ over $\Z$. Since the group $S^{2}(T^*_{sc})^{W}$ is generated by the normalized Killing forms 
{\small \begin{equation}\label{KillE6}
\begin{split}q:=\omega_{1}^2-\omega_{1}\omega_{3}+\omega_{2}^{2}-\omega_{2}\omega_{4}+\omega_{3}^{2}-\omega_{3}\omega_{4}+\omega_{4}^{2}-\omega_{4}\omega_{5}+\omega_{5}^{2}-\omega_{5}\omega_{6}+\omega_{6}^{2} &\,\,\text{ and }\\
q':=\omega_{1}'^2-\omega_{1}'\omega_{3}'+\omega_{2}'^{2}-\omega_{2}'\omega_{4}'+\omega_{3}'^{2}-\omega_{3}'\omega_{4}'+\omega_{4}'^{2}-\omega_{4}'\omega_{5}'+\omega_{5}'^{2}-\omega_{5}'\omega_{6}'+\omega_{6}'^{2},&
\end{split}
\end{equation}}
from the equation $\phi=dq+d'q'$, we obtain 
$$\phi=(\tfrac{2d+2d'}{3})x_{6}'^{2}+\psi,$$
where $\psi$ is a quadratic form with integer coefficients. Therefore,  
\begin{equation}\label{stgenE6}
Q(G)=\{dq+d'q'\,|\,  d+d'\equiv 0 \bmod 3\}.
\end{equation}
Therefore, by (\ref{stgenE6}) and the proof of the previous lemma, we obtain \[\Inv(G)=\Z/2\Z(3q+3q')\oplus \Z/6\Z(q+2q').\] For $n\geq 3$, the same argument shows that \[Q(G)=\{d_{1}q+\cdots d_{n}q_{n}\,|\,  \sum_{i=1}^{n}d_{i}\equiv 0 \bmod 3\},\] where $q_{1},\ldots, q_{n}$ are the corresponding normalized Killing forms. Hence, the result for the indecomposable group follows from Lemma \ref{lem:Esemidec}.

$(2)$ Assume that $n=2$. Then, by \ref{par:char} the character group of the split maximal torus $T$ of $G$ is given by
{\small \[
T^{*}=\{\sum_{i=1}^{7}a_{i}\omega_{i}+\sum_{i=1}^{7}a_{i}'\omega_{i}' \,|\, a_{2}+a_{5}+a_{7}\equiv a_{2}'+a_{5}'+a_{7}' \bmod 2 \}.
\]}
Since the group $S^{2}(T^*_{sc})^{W}$ is generated by the normalized Killing forms {\small \[q:=q_{6}-\omega_{6}\omega_{7}+\omega_{7}^{2} \text{ and }
q':=q_{6}'-\omega_{6}'\omega_{7}'+\omega_{7}'^{2},\]}
where $q_{6}$ and $q_{6}'$ are the normalized Killing forms of $E_{6}$ in (\ref{KillE6}), the same argument as in ($1$) shows
\[
Q(G)=\{dq+d'q'\,|\,  d+d'\equiv 0 \bmod 4\}.
\]
Hence, by Lemma \ref{lem:Esemidec} $\Inv(G)=\Z/3\Z(4q')\oplus \Z/12\Z(q-q')$. For $n\geq 3$, the same argument shows that $Q(G)=\{d_{1}q+\cdots d_{n}q_{n}\,|\,  \sum_{i=1}^{n}d_{i}\equiv 0 \bmod 4\}$, where $q_{1},\ldots, q_{n}$ are the corresponding normalized Killing forms, which comptues the indecomposable group together with Lemma \ref{lem:Esemidec}.
\end{proof}

\begin{remark}
(1) Let $G=(E_{6}^{ad}\times \cdots \times E_{6}^{ad})\times (E_{6}^{sc}\times \cdots \times E_{6}^{sc})$, $n (\geq 0)$ copies of $E_{6}^{ad}$ and $m (\geq 0)$ copies of $E_{6}^{sc}$. It follows by~\eqref{eq:product}, (\ref{stgenE6}) and the proof of Lemma \ref{lem:Esemidec} that $\Inv(G)=(\Z/2\Z)^{\oplus n}\oplus (\Z/6\Z)^{\oplus m}$. Similarly, for the same group $G$ replacing $E_{6}$ by $E_{7}$, we have $\Inv(G)=(\Z/3\Z)^{\oplus n}\oplus (\Z/12\Z)^{\oplus m}$.

(2) Note that the center $\gmu_{3}\times \gmu_{3}$ of $E_{6}^{sc}\times E_{6}^{sc}$ contains two nontrivial ($\neq \gmu_{3}\times 1$, $1\times \gmu_{3}$) central subgroups which is isomorphic to $\gmu_{3}$: a diagonal subgroup and a non-diagonal subgroup. Assume that $\gmu$ is non-diagonal. Then, the character group of $T$ becomes
{\em \begin{equation*}T^{*}=\{\sum_{i=1}^{6}a_{i}\omega_{i}+\sum_{i=1}^{6}a_{i}'\omega_{i}' \,|\, a_{1}+a_{3}'+a_{5}+a_{6}'\equiv a_{1}'+a_{3}+a_{5}'+a_{6} \bmod 3 \}.
\end{equation*}}
In this case, we have the same $Q(G)$ as in (\ref{stgenE6}), thus have the same indecomposable group.\end{remark}

\end{document}